\documentclass{amsart}

\usepackage{amsmath}
\usepackage{amssymb}
\usepackage{amsthm}
\usepackage{graphicx}
\usepackage{hyperref}
\usepackage{enumerate}
\usepackage{subcaption}

\theoremstyle{plain} %
\newtheorem{theorem}{Theorem}[section]

\newtheorem{lemma}[theorem]{Lemma}

\newtheorem{assumption}[theorem]{Assumption}
\newtheorem{remark}[theorem]{Remark}

\title{A rigorous derivation of Haff's law for a periodic two-disk fluid}
\author[A. Grigo]{Alexander Grigo}
\address{University of Oklahoma, Department of Mathematics, Norman OK 73019, USA}
\email{grigo@math.ou.edu}
\thanks{The author's work was partially supported by the NSF grant DMS-1413428.
The author would like to thank
Alexey Korepanov and Ian Melbourne for very stimulating discussions
during the author's visit of the Mathematics Institute at
the University of Warwick, and Leonid Bunimovich for his long-time
support.
The author would like to acknowledge the warm hospitality of
the department of Mathematics at the Southern University of Science
and Technology, China, where this work was finished.
}

\begin{document}

\begin{abstract}
  We derive Haff's cooling law for a periodic fluid consisting
  of two hard disks per unit cell by reducing it to a point particle
  moving inside a Sinai billiard with finite horizon with an
  inelastic collision rule.
  Indeed, our results also apply to general dispersing billiards
  with piece-wise smooth boundary with finite horizon and no cusps.
\end{abstract}

\maketitle

\tableofcontents

\section{Introduction}

One of the central problems in statistical mechanics is to derive
a macroscopic description of a many particle systems based on
a microscopic description. In particular, expressing the
corresponding transport coefficients in terms of the microscopic
interaction model is of interest. Due to its simplicity hard-sphere
models are often used in numerical simulations as well as in
mathematically rigorous investigations. See \cite{MR1805337}
for a collection of surveys on this topic. Mathematically these
models are equivalent to a billiard in a spatial domain of
dimension equal to the number of degrees of freedom of the
hard-sphere model.
In the derivation of transport coefficient one relies on
statistical properties of the microscopic dynamics
\cite{zbMATH00052458}, so that we are interested in
hard-sphere models, i.e. billiards, with hyperbolic dynamics
and good statistical properties.

By now the mathematical analysis of planar hyperbolic billiards is well
developed
\cite{MR597749,MR606459,MR1071936,MR1138952,MR1637655,MR1675363,MR2229799}.
However, for
many particle hard-sphere systems it is extremely challenging
to even show ergodicity. In
\cite{MR1165188,MR2004458,MR3607588}
ergodicity was proven for special model of many interacting particles.
Despite recent progress
\cite{MR1915299,MR2052299,MR2191384,MR2453251,MR2914139}
there is little hope that current techniques will
be able to provide a proof of
statistical properties of multi-particle systems
in the foreseeable future.

The immense technical difficulty is the primary reason why
the mathematically rigorous study of problems of statistical
mechanics, like transport, has in the context of mechanical
models been limited to the study of billiard model in
two dimensional domain \cite{MR1805337}. These minimal
models were used to prove existence of diffusion
\cite{MR606459,MR1138952,MR1149489,MR2349520}
and viscosity
\cite{MR1376436,CHERNOV200037}.
But already a proof of heat conductivity still remains an open challenge
\cite{MR1773043}.

Despite the fact that the statistical properties of
multidimensional billiards remain a challenge, the theory
of planar hyperbolic billiards is developed enough to
also study perturbations of these. Popular choices are
(small) external fields (typically in combination with
a thermostat), which were shown in
\cite{MR1832968,MR1224092,MR2389891}
and
\cite{MR2737493,MR2583573}
to be models that exhibit Ohm's law and the Einstein relation between
the conductivity and diffusion constant.

Modelling dissipative interactions, especially in the context of
kinetic theory of granular media, has been an active area
of research, e.g.
\cite{MR2101911,MR2436466,MR2669629} and references therein.
A novel feature, due to the dissipative nature of the interactions,
is that the interactions slow down the particles. Hence understanding
the cooling process is of interest. In fact, it was shown in
\cite{2150608} that the rate at which the total energy decreases
follows a simple relation, namely as a function time
the inverse of the square-root of the total energy (or rather temperature)
follows a straight line. This relation is commonly known as
Haff's law, and present an additional transport phenomenon, namely
the transport of energy from the system to the ambient environment.

In the present paper we follow the above outlined
philosophy of \cite{MR606459,MR1138952,MR1376436,CHERNOV200037,MR3607588}
and investigate Haff's law in the simplest model, namely two
hard disks on a two-dimensional torus that dissipate energy
due to inelastic collisions. Unfolding the dynamics of the two
disks on the torus to the plane as shown in Fig.~\ref{fig_TwoDisk}
\begin{figure}[ht!]
  \centering
  \begin{subfigure}[b]{0.45\textwidth}
    \centering
    \includegraphics[width=0.95\textwidth]{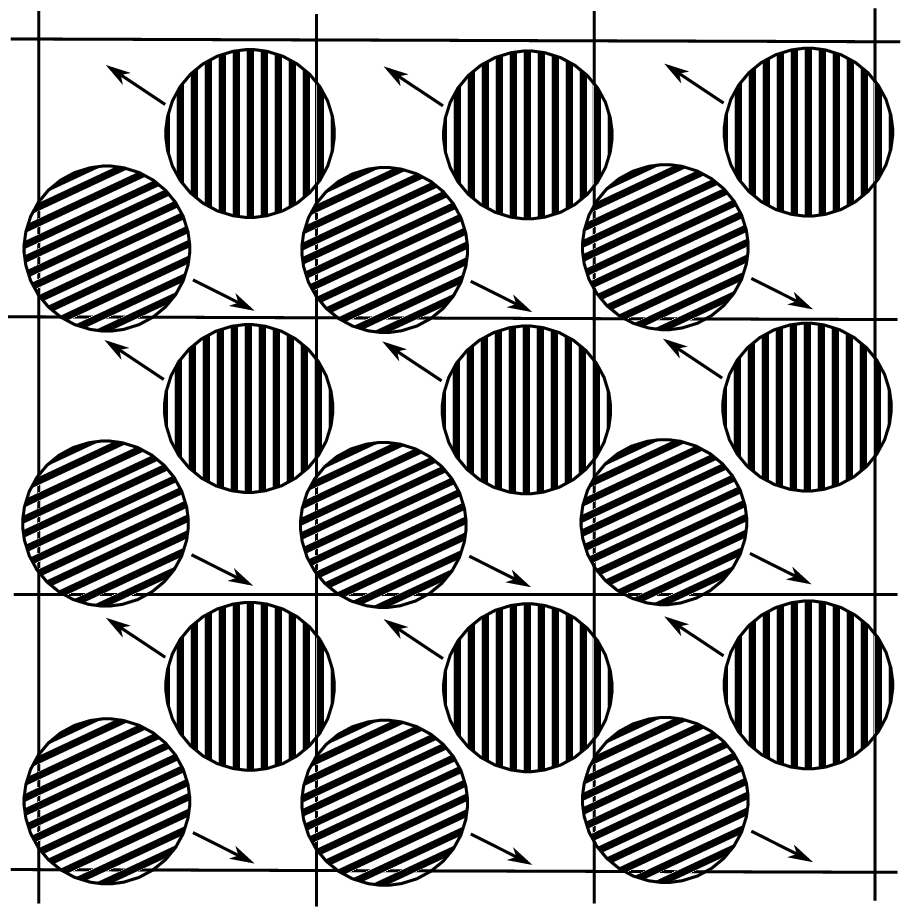}
    \caption{finite horizon}
    \label{fig_TwoDisk_finiteHorizon}
  \end{subfigure}
  \begin{subfigure}[b]{0.45\textwidth}
    \centering
    \includegraphics[width=0.95\textwidth]{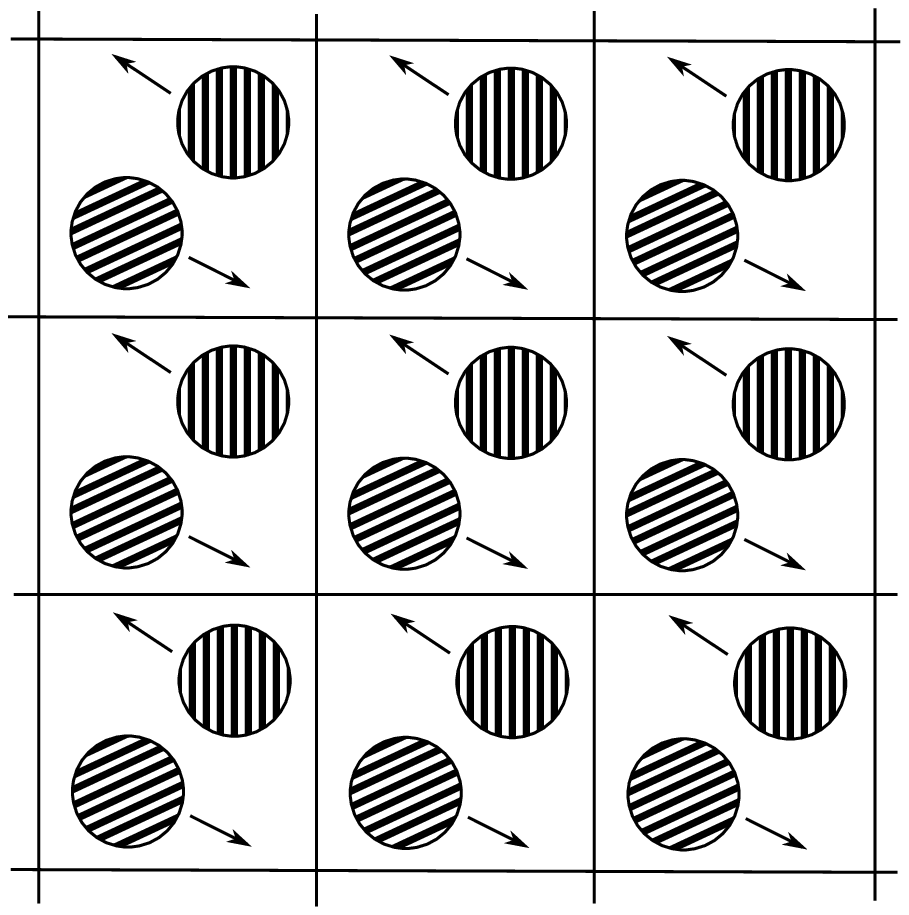}
    \caption{infinite horizon}
    \label{fig_TwoDisk_infiniteHorizon}
  \end{subfigure}
  \caption{The periodic two-disk fluid with finite and infinite horizon.}
  \label{fig_TwoDisk}
\end{figure}
explains why this model corresponds to a periodic two-disk fluid.

As is commonly done
\cite{MR2101911} we will only consider inelastic collisions
that preserve the total momentum of the two particles.
This has the advantage that their center of mass moves with constant
velocity as for the elastic collisions. Therefore, the same
reduction as in the elastic case allows us to only consider the
dynamics of the relative coordinates, which corresponds to a point
particle moving on the two-dimensional torus with a circular
scatterer removed as shown in Fig.~\ref{fig_TwoDisk_unitcell}.
\begin{figure}[ht!]
  \centering
  \begin{subfigure}[b]{0.45\textwidth}
    \centering
    \includegraphics[width=0.75\textwidth]{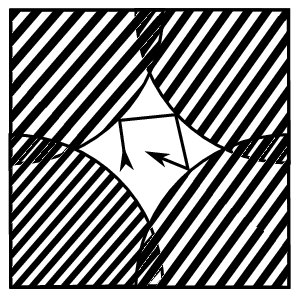}
    \caption{finite horizon}
    \label{fig_TwoDisk_finiteHorizon_unitcell}
  \end{subfigure}
  \begin{subfigure}[b]{0.45\textwidth}
    \centering
    \includegraphics[width=0.75\textwidth]{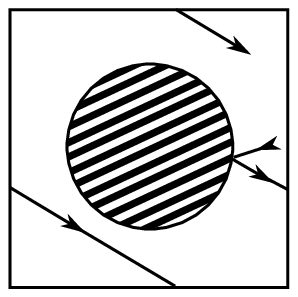}
    \caption{infinite horizon}
    \label{fig_TwoDisk_infiniteHorizon_unitcell}
  \end{subfigure}
  \caption{The relative motion in the
  periodic two-disk fluid with finite and infinite horizon.}
  \label{fig_TwoDisk_unitcell}
\end{figure}
Depending on the size of the two particles relative to the size of the
torus the two particles can pass each other or not, which corresponds
to an infinite or finite horizon Sinai billiard for the dynamics of
the relative coordinates. As in
\cite{MR606459,MR1138952,MR1376436,CHERNOV200037}
we will only consider the case of a finite horizon, i.e. large
particles, to minimize unnecessary technical details and to keep
the presentation of the essentials of the present paper as
clear as possible. Comments on the infinite horizon will be given
in Section~\ref{sect_conclusions}.
In fact, instead of limiting our study to
the very special billiard shown in
Fig.~\ref{fig_TwoDisk_unitcell} we will consider
more general dispersing billiard tables without cusps and
with finite horizon. These would
naturally correspond to the motion of a point particle
in a periodic configuration of scatterers (if the billiard
is on a torus) or inside a closed billiard table with
dispersing boundary components.

In order to make use of the fact that the billiard dynamics
with elastic collisions is known to have strong statistical
properties we naturally will consider the case of small
dissipation, which is also what is often studied in the
setting of kinetic theory, e.g. \cite{MR2101911,MR2436466,MR2669629}.
This naturally presents a fast-slow system, where the usual
billiard coordinates represent the fast variables and the speed (or energy)
of the particle is the slowly changing variable.
Our main result is the following statement of the cooling
process over the relevant time scale, which recovers Haff's law
in the special case where the dissipation mechanism is independent
of the relative speed at the moment of collision, i.e.
a constant restitution coefficient.
A precise formulation is given in Theorem~\ref{thm_avg4}.
\begin{theorem}[Haff's law]
  \label{thm_mainResult}
  Let $Q$ be a dispersing billiard table with piece-wise
  smooth boundary with finite horizon and no cusps.
  In the limit of vanishing dissipation the evolution of the
  speed $\mathsf c$ of the particle
  over s span of time reciprocal to the size of the dissipation
  is uniformly approximated by the solution $\bar{\mathsf c}$ to
  \begin{equation*}
    \frac{d}{dt} \bar{\mathsf c}(t)
    =
    -
    \frac{ |\partial Q| }{ \pi |Q| }
    \, \bar{\mathsf c}(t)^2
    \int_0^{\frac{\pi}{2}}
    q(\bar{\mathsf c}(t)\,\cos\varphi )
    \,\cos^3\varphi\,d\varphi
    \;,
  \end{equation*}
  where the function $q$ models the dissipation
  mechanism.
\end{theorem}

The fact that in our model the dynamics of the fast variables
will depend on the slow variable makes this system a
so-called fully coupled fast-slow system, whose analysis is
generally hard.
Indeed, only very recently fully-coupled fast-slow systems were
investigated \cite{MR2062922,MR2547839,MR2241812}, however, under
the assumption of a smooth dynamics. Since billiards have
singularities those result do not apply right away.
More general dynamical aspects were considered in
\cite{MR3640023}, but not in the fully-coupled setting.
The only directly related work we are aware of is the
study of the motion of a heavy particle colliding with
a light particle \cite{MR2499824}, where as similar fast-slow
system is investigated.

In order to keep the presentation as clear as possible we
will not include lengthy proofs which are just slight modifications
of proofs of similar results. Instead we will point the reader to the
corresponding references, which are primarily
\cite{MR2229799,MR2499824,MR2583573,MR2737493}.

\section{Description of the map}

Let $Q$ denote a domain on the torus, such that
any ray emanating from any point on $\partial Q$ will intersect
$\partial Q$ at a distance that is uniformly bounded away from
$0$ and $\infty$.
Furthermore, we will assume that $\partial Q$ consists of
a finite number of piece-wise smooth (at least $C^3$ smooth) curves
that are convex inwards, also referred to as dispersing.
Consider a particle moving with velocity $v \in {\mathbb R}^2$
along a straight line inside of $Q$ until it reaches a point
$x \in \partial Q$. At such a point of collision the velocity
is instantaneously changed from its pre-collisional value $v_-$
to its post-collisional value $v_+$. Then the particular
continues to move along a straight line with velocity $v_+$
until the next point of collision and so on.

If $v_+$ is obtained from $v_-$ by means of reflection
of $v_-$ about the tangent to the point of
reflection $x \in \partial Q$, i.e.
$
v_+
=
(1 - 2\,{\mathcal N}\,{\mathcal N}^T)\,v_-
$
with ${\mathcal N}$ denoting the unit normal vector to
$x \in \partial Q$, then the dynamics of the
particle is called a dispersing billiard with finite horizon.
We refer to the monograph \cite{MR2229799} and references therein
for a detailed exposition of dynamical properties of such billiards.

In the present work we consider the following model of inelastic
reflection for the expression for $v_+$ in terms of
$v_-$
\begin{equation}
  \label{eqn_def_v_updateRule}
  v_+
  =
  [ 1 - (2 - \eta) \,{\mathcal N}\,{\mathcal N}^T ]\,v_-
  \;,\quad
  \eta
  \equiv
  \eta( -{\mathcal N} \cdot v_-)
  \qquad
  0 \leq \eta < 1
  \;,
\end{equation}
where ${\mathcal N}$ denotes again the unit normal vector
to $\partial Q$ at the point of collision, and
$\eta$ denotes the so-called normal restitution
coefficient, which in general is a function of the normal
velocity $-{\mathcal N} \cdot v_-$.
As detailed in the introduction
the choice of \eqref{eqn_def_v_updateRule}
is motivated by models in the kinetic theory of granular
materials \cite{MR2101911}.

Following standard practice in the theory of billiards
\cite{MR2229799} we decompose the flow into two separate
components. The first is the free flight from one
point on $\partial Q$ to another, the second being the
instantaneous change of the velocity at the moment of collision
with a point on $\partial Q$. The composition of a free
flight with a collision induces a map $\hat{{\mathcal F}}$
\begin{subequations}
  \label{eqn_def_bMapExt}
  \begin{equation}
    \hat{{\mathcal F}}
    \colon
    \hat{{\mathcal M}}
    \to
    \hat{{\mathcal M}}
    \quad\text{where}\quad
    \hat{{\mathcal M}}
    =
    {\mathcal M} \times (0,\infty)
    \quad\text{and}\quad
    {\mathcal M}
    =
    \partial Q \times (-\frac{\pi}{2}, \frac{\pi}{2} )
    \;,
  \end{equation}
  which is the natural generalization of the standard billiard map
  ${\mathcal F}$ defined on ${\mathcal M}$. And we use the usual coordinates
  \begin{equation}
    \hat{x}
    \equiv
    (s, \varphi, \mathsf c)
    \in \hat{{\mathcal M}}
    \quad\text{and}\quad
    x
    \equiv
    (s, \varphi)
    \in {\mathcal M}
    \;,
  \end{equation}
  where $s$ is the arc length parameter along
  $\partial Q$, $\varphi$ of reflection, and $\mathsf c$
  denotes the speed of the particle.
\end{subequations}
Furthermore, we let $\Pi$ denote the projection of $\hat{x}$ onto it
$(s, \varphi)$--component, i.e.
\begin{equation}
  \label{eqn_def_projection}
  \Pi \colon \hat{{\mathcal M}} \to {\mathcal M}
  \;,\quad
  \Pi(s, \varphi, \mathsf c)
  =
  (s, \varphi)
  \;.
\end{equation}

A natural representation of
$
(s_1, \varphi_1, \mathsf c_1)
=
\hat{{\mathcal F}}(s_0, \varphi_0, \mathsf c_0)
$
is given by first applying the standard billiard map
\begin{subequations}
  \label{eqn_def_bMapExt_rep}
  \begin{equation}
    (s_1, \tilde{\varphi}_1)
    =
    {\mathcal F}( s_0, \varphi_0)
  \end{equation}
  and then determine the angle $\varphi_1$ and speed $\mathsf c_1$
  according to \eqref{eqn_def_v_updateRule}
  \begin{equation}
    \label{eqn_updateRule_phi_c_implicit}
    \mathsf c_1 \,\cos\varphi_1
    =
    [ 1 - \eta(\mathsf c_0\,\cos\tilde{\varphi}_1 ) ]
    \,\mathsf c_0\,\cos\tilde{\varphi}_1
    \quad\text{and}\quad
    \mathsf c_1 \,\sin\varphi_1
    =
    \mathsf c_0\,\sin\tilde{\varphi}_1
  \end{equation}
\end{subequations}
i.e.
\begin{equation}
  \label{eqn_updateRule_phi_c_explicit}
  \begin{split}
    \varphi_1
    &=
    \arctan
    \frac{ \tan\tilde{\varphi}_1
    }{ 1 - \eta(\mathsf c_0\,\cos\tilde{\varphi}_1 ) }
    \\
    \mathsf c_1
    &=
    \mathsf c_0
    \,\sqrt{
    [ 1 - \eta(\mathsf c_0\,\cos\tilde{\varphi}_1 ) ]^2
    \,\cos^2\tilde{\varphi}_1
    +
    \sin^2\tilde{\varphi}_1
    }
    \;.
  \end{split}
\end{equation}
With this notation, introduce the mappings
\begin{equation}
  \label{eqn_def_mapPert}
  P(\tilde{\varphi}_1, \mathsf c_0)
  =
  (\varphi_1, \mathsf c_1)
  \quad\text{and}\quad
  \hat{P}(s_1, \tilde{\varphi}_1, \mathsf c_0)
  =
  (s_1, \varphi_1, \mathsf c_1)
  \;,
\end{equation}
and
\begin{equation}
  \label{eqn_def_bMapExtZero}
  \hat{{\mathcal F}}_0(s, \varphi, \mathsf c)
  =
  ( {\mathcal F}(s, \varphi), \mathsf c)
  \;,
\end{equation}
so that
\begin{equation*}
  \hat{{\mathcal F}}(s, \varphi, \mathsf c)
  =
  \hat{P}({\mathcal F}(s, \varphi), \mathsf c)
  =
  \hat{P} \circ \hat{{\mathcal F}}_0(s, \varphi, \mathsf c)
\end{equation*}
holds. Throughout the paper we will frequently make use of
\begin{equation}
  \label{eqn_def_dissCoeff_D_DD}
  \eta_1(w) = w\,\eta'(w)
  \;,\quad
  \eta_2(w) = w^2\,\eta''(w)
  \quad\text{for all}\quad w \geq 0
\end{equation}
to shorten the notation.

Using the two-stage representation \eqref{eqn_def_bMapExt_rep}
of $\hat{{\mathcal F}}$ we readily obtain an expression for
$\mathrm{D}\hat{{\mathcal F}}$ in terms of the derivative of $\mathrm{D}{\mathcal F}$, \cite{MR2229799},
and the derivative of the mapping
$
(\tilde{\varphi}_1, \mathsf c_0)
\mapsto
(\varphi_1, \mathsf c_1)
$.
The expression for
$\mathrm{D}{\mathcal F}(s_0, \varphi_0)$
is given by \cite{MR2229799}
\begin{subequations}
  \label{eqn_DbMapExt}
  \begin{equation}
    \label{eqn_DbMapExt_DbMap}
    \mathrm{D}{\mathcal F}(s_0, \varphi_0)
    =
    -\frac{1}{\cos\tilde{\varphi}_1}
    \begin{pmatrix}
      \tau_{01}\,{\mathcal K}_0 + \cos\varphi_0
      &
      \tau_{01}
      \\
      \tau_{01}\,{\mathcal K}_0\,{\mathcal K}_1
      +
      {\mathcal K}_0\,\cos\tilde{\varphi}_1
      +
      {\mathcal K}_1\,\cos\varphi_0
      &
      \tau_{01}\,{\mathcal K}_1 + \cos\tilde{\varphi}_1
    \end{pmatrix}
    \;,
  \end{equation}
  where $\tau_{01}$ denotes the length of the free path, and
  ${\mathcal K}_0$ and ${\mathcal K}_1$ denote the
  curvature of the boundary $\partial Q$ at
  $s_0$ and $s_1$, respectively.
  By differentiating \eqref{eqn_updateRule_phi_c_implicit}
  we obtain
  \begin{align*}
    \begin{pmatrix}
      -\mathsf c_1 \,\sin\varphi_1
      &
      \cos\varphi_1
      \\
      \mathsf c_1 \,\cos\varphi_1
      &
      \sin\varphi_1
    \end{pmatrix}
    \begin{pmatrix}
      d\varphi_1 \\
      d\mathsf c_1
    \end{pmatrix}
    &=
    \begin{pmatrix}
      -
      \mathsf c_0\,\sin\tilde{\varphi}_1
      \,[1 - \eta - \eta_1]
      &
      \cos\tilde{\varphi}_1
      \,[1 - \eta - \eta_1]
      \\
      \mathsf c_0\,\cos\tilde{\varphi}_1
      &
      \sin\tilde{\varphi}_1
    \end{pmatrix}
    \begin{pmatrix}
      d\tilde{\varphi}_1 \\
      d\mathsf c_0
    \end{pmatrix}
  \end{align*}
  where we used the short-hand notation
  $
  \eta \equiv \eta(\mathsf c_0\,\cos\tilde{\varphi}_1 )
  $
  and
  $
  \eta_1 \equiv \eta_1(\mathsf c_0\,\cos\tilde{\varphi}_1 )
  $a.
  Therefore it follows that
  \begin{equation}
    \label{eqn_DbMapExt_Dpert}
    \begin{split}
      \frac{ \partial\varphi_1 }{ \partial\tilde{\varphi}_1 }
      &=
      \frac{ 1 - \eta }{
      ( 1 - \eta )^2 \,\cos^2\tilde{\varphi}_1
      +
      \sin^2\tilde{\varphi}_1
      }
      -
      \sin^2\varphi_1
      \,\eta_1
      =
      \frac{ \sin\varphi_1 \,\cos\varphi_1
      }{ \sin\tilde{\varphi}_1 \,\cos\tilde{\varphi}_1 }
      -
      \sin^2\varphi_1
      \,\eta_1
      \\
      \frac{ \partial\varphi_1 }{ \partial\mathsf c_0 }
      &=
      \frac{1}{ \mathsf c_0 }
      \,\frac{ \sin\tilde{\varphi}_1 \,\cos\tilde{\varphi}_1
      }{
      ( 1 - \eta )^2 \,\cos^2\tilde{\varphi}_1
      +
      \sin^2\tilde{\varphi}_1
      }
      \,\eta_1
      =
      \frac{ \mathsf c_0 }{ \mathsf c_1^2 }
      \,\sin\tilde{\varphi}_1
      \,\cos\tilde{\varphi}_1
      \,\eta_1
      \\
      \frac{ \partial\mathsf c_1 }{ \partial\tilde{\varphi}_1 }
      &=
      \mathsf c_0
      \,\sin\tilde{\varphi}_1
      \,\cos\varphi_1
      \,\Big[
      \frac{ ( 2 - \eta )\,\eta }{ 1 - \eta }
      +
      \eta_1
      \Big]
      \\
      \frac{ \partial\mathsf c_1 }{ \partial\mathsf c_0 }
      &=
      \frac{ \mathsf c_1 }{ \mathsf c_0 }
      -
      \cos\varphi_1
      \,\cos\tilde{\varphi}_1
      \,\eta_1
      =
      \frac{ \sin\tilde{\varphi}_1 }{ \sin\varphi_1 }
      -
      \cos\varphi_1
      \,\cos\tilde{\varphi}_1
      \,\eta_1
    \end{split}
  \end{equation}
  where we made use of \eqref{eqn_updateRule_phi_c_explicit}
  to simplify.
  In particular,
  \begin{equation}
    \mathrm{D}\hat{{\mathcal F}}(s_0, \varphi_0, \mathsf c_0)
    =
    \begin{pmatrix}
      1 & 0 & 0 \\
      0 &
      \frac{ \partial\varphi_1 }{ \partial\tilde{\varphi}_1 }
      &
      \frac{ \partial\varphi_1 }{ \partial\mathsf c_0 }
      \\
      0 &
      \frac{ \partial\mathsf c_1 }{ \partial\tilde{\varphi}_1 }
      & 
      \frac{ \partial\mathsf c_1 }{ \partial\mathsf c_0 }
    \end{pmatrix}
    \begin{pmatrix}
      \mathrm{D}{\mathcal F}(s_0, \varphi_0)
      &
      \begin{matrix}
        0 \\
        0
      \end{matrix}
      \\
      \begin{matrix}
        0 & 0
      \end{matrix}
      & 1
    \end{pmatrix}
  \end{equation}
\end{subequations}
for the explicit expression for the derivative of $\hat{{\mathcal F}}$.

\section{Invariant cone fields}
\label{sect_cones}

Throughout we will denote by
${\mathcal V}_{\mathrm{min}}$, ${\mathcal V}_{\mathrm{max}}$, $\kappa$
three parameters that will always be assumed to satisfy
\begin{equation*}
  0 < {\mathcal V}_{\mathrm{min}} < {\mathcal V}_{\mathrm{max}} \leq \infty
  \;,\quad
  0 \leq \kappa < \infty
  \;.
\end{equation*}
For any such choice we define the cone field
\begin{equation}
  \label{eqn_def_uConeExt}
  \hat{{\mathcal C}}^u_{\hat{x}}
  =
  \Big\{
  d\hat{x} \operatorname{:}
  {\mathcal V}_{\mathrm{min}}
  \leq
  \frac{d\varphi}{ds}
  \leq
  {\mathcal V}_{\mathrm{max}}
  \quad\text{and}\quad
  \Big|
  \frac{d\mathsf c}{ \mathsf c\,\cos\varphi\,ds }
  \Big|
  \leq
  \kappa
  \Big\}
  \;,
\end{equation}
and the corresponding
\begin{equation}
  \label{eqn_def_uConeExt_projection}
  {\mathcal C}^u_{x}
  =
  \Pi \hat{{\mathcal C}}^u_{ \Pi(\hat{x})}
  =
  \Big\{
  dx \operatorname{:}
  {\mathcal V}_{\mathrm{min}}
  \leq
  \frac{d\varphi}{ds}
  \leq
  {\mathcal V}_{\mathrm{max}}
  \Big\}
\end{equation}
projection of $\hat{{\mathcal C}}^u_{\hat{x}}$ for any
$\hat{x}$ with $\Pi(\hat{x}) = x$.
Furthermore, we will assume throughout that
\begin{equation}
  \eta_{\mathrm{max}}
  =
  \sup_{\mathsf c \geq 0} \eta(\mathsf c)
  <\infty
  \;,\quad
  \eta_{1,\mathrm{max}}
  =
  \sup_{\mathsf c \geq 0} |\eta_1(\mathsf c)|
  <\infty
  \;,\quad
  \eta_{2,\mathrm{max}}
  =
  \sup_{\mathsf c \geq 0} |\eta_2(\mathsf c)|
  <\infty
  \;.
\end{equation}
Clearly $0 \leq \eta_{\mathrm{max}} \leq 1$, and for any
given billiard table $Q$ the two conditions
\begin{equation}
  \tag{C}
  \begin{split}
      \eta_{\mathrm{max}} + \eta_{1,\mathrm{max}}
      &<
      ( 1 - \eta_{\mathrm{max}} )\, \tau_{\mathrm{min}}\,{\mathcal K}_{\mathrm{min}}
      \\
      \frac{ \eta_{1,\mathrm{max}} }{ ( 1 - \eta_{\mathrm{max}})^{\frac{5}{2}} }
      \,\frac{
      \frac{ 2 - \eta_{\mathrm{max}} }{ 1 - \eta_{\mathrm{max}} }\,\eta_{\mathrm{max}}
      +
      \eta_{1,\mathrm{max}}
      }{
      \tau_{\mathrm{min}}\,{\mathcal K}_{\mathrm{min}}
      -
      \frac{ \eta_{\mathrm{max}} + \eta_{1,\mathrm{max}} }{ 1 - \eta_{\mathrm{max}} }
      }
      &<
      ( 1 - \eta_{\mathrm{max}} - \eta_{1,\mathrm{max}} )
      \,\frac{ \tau_{\mathrm{min}}\, {\mathcal K}_{\mathrm{min}}
      }{ 1 + \tau_{\mathrm{min}}\, {\mathcal K}_{\mathrm{max}} }
  \end{split}
  \label{eqn_uConeCondition}
\end{equation}
are satisfied for all small enough values of $\eta_{\mathrm{max}}$ and
$\eta_{1,\mathrm{max}}$. Here we assume that $\tau_{\mathrm{min}} >0$, which will
be a temporary assumption we make to simplify the exposition.
We will comment on the situation where $\tau =0$ right
before stating Theorem~\ref{thm_avg3}.

\begin{lemma}[Invariance of $\hat{{\mathcal C}}^u$ and ${\mathcal C}^u$]
  \label{lem_uCones_invariance}
  Suppose $\eta_{\mathrm{max}}$ and $\eta_{1,\mathrm{max}}$ satisfy
  \eqref{eqn_uConeCondition}.
  Then for any choice of the parameters
  ${\mathcal V}_{\mathrm{min}}$, ${\mathcal V}_{\mathrm{max}}$, $\kappa$
  such that
  \begin{align*}
    \kappa
    &\geq
    \frac{
    1 + \tau_{\mathrm{min}}\,{\mathcal K}_{\mathrm{min}}
    }{
    \sqrt{ 1 - \eta_{\mathrm{max}} }
    }
    \,\frac{
    \frac{ 2 - \eta_{\mathrm{max}} }{ 1 - \eta_{\mathrm{max}} }\,\eta_{\mathrm{max}}
    +
    \eta_{1,\mathrm{max}}
    }{
    \tau_{\mathrm{min}}\,{\mathcal K}_{\mathrm{min}}
    -
    \frac{ \eta_{\mathrm{max}} + \eta_{1,\mathrm{max}} }{ 1 - \eta_{\mathrm{max}} }
    }
    \,\Big( {\mathcal K}_{\mathrm{max}} + \frac{1}{ \tau_{\mathrm{min}} } \Big)
    \\
    {\mathcal V}_{\mathrm{min}}
    &\leq
    ( 1 - \eta_{\mathrm{max}} - \eta_{1,\mathrm{max}} )
    \,{\mathcal K}_{\mathrm{min}}
    -
    \frac{ \eta_{1,\mathrm{max}} }{ ( 1 - \eta_{\mathrm{max}} )^2 }
    \,\frac{ \kappa }{ 1 + \tau_{\mathrm{min}}\,{\mathcal K}_{\mathrm{min}} }
    \\
    {\mathcal V}_{\mathrm{max}}
    &\geq
    \Big( \frac{1 }{ 1 - \eta_{\mathrm{max}} } + \eta_{1,\mathrm{max}} \Big)
    \,\Big( {\mathcal K}_{\mathrm{max}} + \frac{1}{ \tau_{\mathrm{min}} } \Big)
    +
    \frac{ \eta_{1,\mathrm{max}} }{ ( 1 - \eta_{\mathrm{max}} )^2 }
    \,\frac{ \kappa }{ 1 + \tau_{\mathrm{min}}\,{\mathcal K}_{\mathrm{min}} }
  \end{align*}
  the corresponding cone fields $\hat{{\mathcal C}}^u$ and ${\mathcal C}^u$
  are invariant under $\hat{{\mathcal F}}$ and ${\mathcal F}$, respectively.
\end{lemma}
\begin{proof}
  The invariance property of ${\mathcal C}^u$ is well-known \cite{MR2229799}.
  Let $\hat{x}_0$ be arbitrary, and suppose that
  $d\hat{x}_0 \in \hat{{\mathcal C}}^u_{\hat{x}_0}$. Denote by
  $\hat{x}_1$ the point $\hat{{\mathcal F}}(\hat{x}_0)$, and
  denote by $d\hat{x}_1$ the vector
  $\mathrm{D}\hat{{\mathcal F}}(\hat{x}_0)\,d\hat{x}_0$.
  Fix $0 \leq {\mathcal V}_0 < {\mathcal V}_{\mathrm{max}} \leq \infty$ and
  $0 \leq \kappa < \infty$.
  The explicit expression \eqref{eqn_DbMapExt}
  for the derivative of $\hat{{\mathcal F}}$ yields
  \begin{align*}
    \frac{ d\varphi_1 }{ ds_1 }
    &=
    \frac{ \partial\varphi_1 }{ \partial\tilde{\varphi}_1 }
    \,\Big[
    {\mathcal K}_1
    +
    \frac{ \cos\tilde{\varphi}_1 }{
    \tau_{01}
    +
    \frac{ \cos\varphi_0
    }{ {\mathcal K}_0 + \frac{ d\varphi_0 }{ ds_0 } }
    }
    \Big]
    -
    \frac{
    \mathsf c_0\,\cos\tilde{\varphi}_1
    \,\frac{ \partial\varphi_1 }{ \partial\mathsf c_0 }
    }{
    1
    +
    \tau_{01}
    \,\frac{
    {\mathcal K}_0 + \frac{ d\varphi_0 }{ ds_0 }
    }{ \cos\varphi_0 }
    }
    \,\frac{ d\mathsf c_0 }{
    \mathsf c_0\,\cos\varphi_0\,ds_0 }
    \\
    \frac{ d\mathsf c_1 }{
    \mathsf c_1\,\cos\varphi_1\,ds_1 }
    &=
    \frac{1}{ \mathsf c_1\,\cos\varphi_1 }
    \,\frac{ \partial\mathsf c_1 }{ \partial\tilde{\varphi}_1 }
    \,\Big[
    {\mathcal K}_1
    +
    \frac{ \cos\tilde{\varphi}_1 }{
    \tau_{01}
    +
    \frac{ \cos\varphi_0
    }{ {\mathcal K}_0 + \frac{ d\varphi_0 }{ ds_0 } }
    }
    \Big]
    -
    \frac{
    \frac{\mathsf c_0}{\mathsf c_1}
    \frac{ \partial\mathsf c_1 }{ \partial\mathsf c_0 }
    \,\frac{ \cos\tilde{\varphi}_1 }{ \cos\varphi_1 }
    }{
    1
    +
    \tau_{01}
    \,\frac{
    {\mathcal K}_0 + \frac{ d\varphi_0 }{ ds_0 }
    }{ \cos\varphi_0 }
    }
    \,\frac{ d\mathsf c_0 }{
    \mathsf c_0\,\cos\varphi_0\,ds_0 }
    \;,
  \end{align*}
  which allows us to compute the image $d\hat{x}_1$
  of $d\hat{x}_0$ under $\mathrm{D}\hat{{\mathcal F}}$. With
  $
  {\mathcal V}_{\mathrm{min}}
  \leq
  \tfrac{d\varphi_0}{ds_0}
  \leq
  {\mathcal V}_{\mathrm{max}}
  $
  and
  $
  |
  \tfrac{ d\mathsf c_0 }{
  \mathsf c_0\,\cos\varphi_0\,ds_0 }
  |
  \leq \kappa
  $
  it thus follows that
  \begin{align*}
    \frac{ d\varphi_1 }{ ds_1 }
    &\leq
    \Big|
    \frac{ \partial\varphi_1 }{ \partial\tilde{\varphi}_1 }
    \Big|
    \,\Big( {\mathcal K}_{\mathrm{max}} + \frac{1}{ \tau_{\mathrm{min}} } \Big)
    +
    \frac{
    |
    \mathsf c_0\,\cos\tilde{\varphi}_1
    \,\frac{ \partial\varphi_1 }{ \partial\mathsf c_0 }
    |
    }{ 1 + \tau_{\mathrm{min}}\,{\mathcal K}_{\mathrm{min}} }
    \,\kappa
    \\
    \frac{ d\varphi_1 }{ ds_1 }
    &\geq
    \Big|
    \frac{ \partial\varphi_1 }{ \partial\tilde{\varphi}_1 }
    \Big|
    \,{\mathcal K}_{\mathrm{min}}
    -
    \frac{
    |
    \mathsf c_0\,\cos\tilde{\varphi}_1
    \,\frac{ \partial\varphi_1 }{ \partial\mathsf c_0 }
    |
    }{ 1 + \tau_{\mathrm{min}}\,{\mathcal K}_{\mathrm{min}} }
    \,\kappa
    \\
    \Big|
    \frac{ d\mathsf c_1 }{
    \mathsf c_1\,\cos\varphi_1\,ds_1 }
    \Big|
    &\leq
    \Big|
    \frac{1}{ \mathsf c_1\,\cos\varphi_1 }
    \,\frac{ \partial\mathsf c_1 }{ \partial\tilde{\varphi}_1 }
    \Big|
    \,\Big( {\mathcal K}_{\mathrm{max}} + \frac{1}{ \tau_{\mathrm{min}} } \Big)
    +
    \frac{
    |
    \frac{\mathsf c_0}{\mathsf c_1}
    \frac{ \partial\mathsf c_1 }{ \partial\mathsf c_0 }
    \,\frac{ \cos\tilde{\varphi}_1 }{ \cos\varphi_1 }
    |
    }{ 1 + \tau_{\mathrm{min}}\,{\mathcal K}_{\mathrm{min}} }
    \,\kappa
    \;.
  \end{align*}
  Using the explicit expressions
  \eqref{eqn_DbMapExt_Dpert}
  for the various derivatives we see that
  \begin{align*}
    1 - \eta_{\mathrm{max}} - \eta_{1,\mathrm{max}}
    &\leq
    \Big|
    \frac{ \partial\varphi_1 }{ \partial\tilde{\varphi}_1 }
    \Big|
    \leq
    \frac{1 }{ 1 - \eta_{\mathrm{max}} } + \eta_{1,\mathrm{max}}
    \\
    \Big|
    \mathsf c_0\,\cos\tilde{\varphi}_1
    \,\frac{ \partial\varphi_1 }{ \partial\mathsf c_0 }
    \Big|
    &\leq
    \frac{ \eta_{1,\mathrm{max}} }{ ( 1 - \eta_{\mathrm{max}} )^2 }
    \\
    \Big|
    \frac{1}{ \mathsf c_1\,\cos\varphi_1 }
    \,\frac{ \partial\mathsf c_1 }{ \partial\tilde{\varphi}_1 }
    \Big|
    &\leq
    \frac{
    \frac{ ( 2 - \eta_{\mathrm{max}} )\,\eta_{\mathrm{max}} }{ 1 - \eta_{\mathrm{max}} }
    +
    \eta_{1,\mathrm{max}}
    }{ \sqrt{ 1 - \eta_{\mathrm{max}} } }
    \\
    \Big|
    \frac{\mathsf c_0}{\mathsf c_1}
    \frac{ \partial\mathsf c_1 }{ \partial\mathsf c_0 }
    \,\frac{ \cos\tilde{\varphi}_1 }{ \cos\varphi_1 }
    \Big|
    &\leq
    \frac{ 1 + \eta_{1,\mathrm{max}} }{ 1 - \eta_{\mathrm{max}} }
  \end{align*}
  and hence
  \begin{align*}
    \frac{ d\varphi_1 }{ ds_1 }
    &\leq
    \Big( \frac{1 }{ 1 - \eta_{\mathrm{max}} } + \eta_{1,\mathrm{max}} \Big)
    \,\Big( {\mathcal K}_{\mathrm{max}} + \frac{1}{ \tau_{\mathrm{min}} } \Big)
    +
    \frac{ \eta_{1,\mathrm{max}} }{ ( 1 - \eta_{\mathrm{max}} )^2 }
    \,\frac{ \kappa }{ 1 + \tau_{\mathrm{min}}\,{\mathcal K}_{\mathrm{min}} }
    \\
    \frac{ d\varphi_1 }{ ds_1 }
    &\geq
    ( 1 - \eta_{\mathrm{max}} - \eta_{1,\mathrm{max}} )
    \,{\mathcal K}_{\mathrm{min}}
    -
    \frac{ \eta_{1,\mathrm{max}} }{ ( 1 - \eta_{\mathrm{max}} )^2 }
    \,\frac{ \kappa }{ 1 + \tau_{\mathrm{min}}\,{\mathcal K}_{\mathrm{min}} }
    \\
    \Big|
    \frac{ d\mathsf c_1 }{
    \mathsf c_1\,\cos\varphi_1\,ds_1 }
    \Big|
    &\leq
    \frac{
    \frac{ ( 2 - \eta_{\mathrm{max}} )\,\eta_{\mathrm{max}} }{ 1 - \eta_{\mathrm{max}} }
    +
    \eta_{1,\mathrm{max}}
    }{ \sqrt{ 1 - \eta_{\mathrm{max}} } }
    \,\Big( {\mathcal K}_{\mathrm{max}} + \frac{1}{ \tau_{\mathrm{min}} } \Big)
    +
    \frac{ 1 + \eta_{1,\mathrm{max}} }{ 1 - \eta_{\mathrm{max}} }
    \,\frac{ \kappa }{ 1 + \tau_{\mathrm{min}}\,{\mathcal K}_{\mathrm{min}} }
    \;.
  \end{align*}
  Therefore, it is sufficient for the invariance of $\hat{{\mathcal C}}^u$
  to have $0 \leq \kappa$, and
  $0 \leq {\mathcal V}_{\mathrm{min}} \leq {\mathcal V}_{\mathrm{max}} \leq \infty$
  such that
  \begin{align*}
    {\mathcal V}_{\mathrm{max}}
    &\geq
    \Big( \frac{1 }{ 1 - \eta_{\mathrm{max}} } + \eta_{1,\mathrm{max}} \Big)
    \,\Big( {\mathcal K}_{\mathrm{max}} + \frac{1}{ \tau_{\mathrm{min}} } \Big)
    +
    \frac{ \eta_{1,\mathrm{max}} }{ ( 1 - \eta_{\mathrm{max}} )^2 }
    \,\frac{ \kappa }{ 1 + \tau_{\mathrm{min}}\,{\mathcal K}_{\mathrm{min}} }
    \\
    0
    \leq
    {\mathcal V}_{\mathrm{min}}
    &\leq
    ( 1 - \eta_{\mathrm{max}} - \eta_{1,\mathrm{max}} )
    \,{\mathcal K}_{\mathrm{min}}
    -
    \frac{ \eta_{1,\mathrm{max}} }{ ( 1 - \eta_{\mathrm{max}} )^2 }
    \,\frac{ \kappa }{ 1 + \tau_{\mathrm{min}}\,{\mathcal K}_{\mathrm{min}} }
    \\
    \kappa
    &\geq
    \frac{
    \frac{ ( 2 - \eta_{\mathrm{max}} )\,\eta_{\mathrm{max}} }{ 1 - \eta_{\mathrm{max}} }
    +
    \eta_{1,\mathrm{max}}
    }{ \sqrt{ 1 - \eta_{\mathrm{max}} } }
    \,\Big( {\mathcal K}_{\mathrm{max}} + \frac{1}{ \tau_{\mathrm{min}} } \Big)
    +
    \frac{ 1 + \eta_{1,\mathrm{max}} }{ 1 - \eta_{\mathrm{max}} }
    \,\frac{ \kappa }{ 1 + \tau_{\mathrm{min}}\,{\mathcal K}_{\mathrm{min}} }
    \;.
  \end{align*}
  By our first assumption
  \begin{equation*}
    0
    <
    1
    -
    \frac{ 1 + \eta_{1,\mathrm{max}} }{ 1 - \eta_{\mathrm{max}} }
    \,\frac{ 1 }{ 1 + \tau_{\mathrm{min}}\,{\mathcal K}_{\mathrm{min}} }
  \end{equation*}
  so that the condition on $\kappa$ takes on the
  equivalent form
  \begin{align*}
    \kappa
    &\geq
    \frac{
    1 + \tau_{\mathrm{min}}\,{\mathcal K}_{\mathrm{min}}
    }{
    \sqrt{ 1 - \eta_{\mathrm{max}} }
    }
    \,\frac{
    \frac{ 2 - \eta_{\mathrm{max}} }{ 1 - \eta_{\mathrm{max}} }\,\eta_{\mathrm{max}}
    +
    \eta_{1,\mathrm{max}}
    }{
    \tau_{\mathrm{min}}\,{\mathcal K}_{\mathrm{min}}
    -
    \frac{ \eta_{\mathrm{max}} + \eta_{1,\mathrm{max}} }{ 1 - \eta_{\mathrm{max}} }
    }
    \,\Big( {\mathcal K}_{\mathrm{max}} + \frac{1}{ \tau_{\mathrm{min}} } \Big)
    \;.
  \end{align*}
  In particular
  \begin{align*}
    \frac{ \eta_{1,\mathrm{max}} }{ ( 1 - \eta_{\mathrm{max}} )^2 }
    \,\frac{ \kappa }{ 1 + \tau_{\mathrm{min}}\,{\mathcal K}_{\mathrm{min}} }
    &\geq
    \frac{ \eta_{1,\mathrm{max}} }{ ( 1 - \eta_{\mathrm{max}})^{\frac{5}{2}} }
    \,\frac{
    \frac{ 2 - \eta_{\mathrm{max}} }{ 1 - \eta_{\mathrm{max}} }\,\eta_{\mathrm{max}}
    +
    \eta_{1,\mathrm{max}}
    }{
    \tau_{\mathrm{min}}\,{\mathcal K}_{\mathrm{min}}
    -
    \frac{ \eta_{\mathrm{max}} + \eta_{1,\mathrm{max}} }{ 1 - \eta_{\mathrm{max}} }
    }
    \,\Big( {\mathcal K}_{\mathrm{max}} + \frac{1}{ \tau_{\mathrm{min}} } \Big)
    \;.
  \end{align*}
  On the other hand, for ${\mathcal V}_{\mathrm{min}}$ to exists we also
  need to have
  \begin{align*}
    \frac{ \eta_{1,\mathrm{max}} }{ ( 1 - \eta_{\mathrm{max}} )^2 }
    \,\frac{ \kappa }{ 1 + \tau_{\mathrm{min}}\,{\mathcal K}_{\mathrm{min}} }
    <
    ( 1 - \eta_{\mathrm{max}} - \eta_{1,\mathrm{max}} )
    \,{\mathcal K}_{\mathrm{min}}
  \end{align*}
  which is possible due to our second assumption, and thus the three
  conditions stated Lemma~\ref{lem_uCones_invariance}
  can be simultaneously satisfied.
\end{proof}

For any $d\hat{x} \in \hat{{\mathcal C}}^u_{\hat{x}}$,
$dx \in {\mathcal C}^u_{x}$
we define its
Euclidean norm and its adapted norm by
\begin{equation}
  \label{eqn_def_norm_uVectors}
  \begin{split}
  \left\Arrowvert \, d\hat{x} \, \right\Arrowvert
  &=
  \sqrt{ ds^2 + d\varphi^2 + d\mathsf c^2 }
  \;,\quad
  \left| \, d\hat{x} \, \right|_*
  =
  \cos\varphi\,| ds |
  \\
  \left\Arrowvert \, dx \, \right\Arrowvert
  &=
  \sqrt{ ds^2 + d\varphi^2 }
  \;,\quad
  \left| \, dx \, \right|_*
  =
  \cos\varphi\,| ds |
  \end{split}
  \;,
\end{equation}
respectively. In particular we have
\begin{equation}
  \label{eqn_norm_uVectors_equivalence}
  \begin{split}
    \frac{ \sqrt{ 1 + {\mathcal V}_{\mathrm{min}}^2 } }{ \cos\varphi }
    \,\left| \, d\hat{x} \, \right|_*
    \leq
    \left\Arrowvert \, d\hat{x} \, \right\Arrowvert
    &\leq
    \frac{
    \sqrt{ 1 + {\mathcal V}_{\mathrm{max}}^2 + \kappa^2\,\mathsf c^2\,\cos^2\varphi }
    }{ \cos\varphi }
    \,\left| \, d\hat{x} \, \right|_*
    \\
    \frac{ \sqrt{ 1 + {\mathcal V}_{\mathrm{min}}^2 } }{ \cos\varphi }
    \,\left| \, dx \, \right|_*
    \leq
    \left\Arrowvert \, dx \, \right\Arrowvert
    &\leq
    \frac{
    \sqrt{ 1 + {\mathcal V}_{\mathrm{max}}^2 }
    }{ \cos\varphi }
    \,\left| \, dx \, \right|_*
    \;,
  \end{split}
  \;,
\end{equation}
which shows that the two norms are (locally) equivalent on
$\hat{{\mathcal C}}^u$. 
These are the natural generalizations of the
corresponding standard concepts in the theory of dispersing
billiards \cite{MR2229799}. The following result is
an immediate consequence of the definition of 
$\left| \, \cdot \, \right|_*$.

\begin{lemma}[Uniform expansion of $\hat{{\mathcal C}}^u$ and ${\mathcal C}^u$]
  \label{lem_uCones_expansion}
  Under the assumptions of Lemma~\ref{lem_uCones_invariance}
  \begin{equation*}
    \left| \,  \mathrm{D}\hat{{\mathcal F}}(\hat{x})\,d\hat{x}  \, \right|_*
    \geq
    \Lambda
    \,\left| \, d\hat{x} \, \right|_*
    \quad\text{and}\quad
    \left| \,  \mathrm{D}{\mathcal F}(x)\,dx  \, \right|_*
    \geq
    \Lambda
    \,\left| \, dx \, \right|_*
  \end{equation*}
  with
  \begin{equation}
    \label{eqn_def_uConeRate}
    \Lambda
    =
    ( 1 - \eta_{\mathrm{max}} )
    \,[ 1 + \tau_{\mathrm{min}} \,( {\mathcal K}_{\mathrm{min}} + {\mathcal V}_{\mathrm{min}} ) ]
  \end{equation}
  holds for
  any $\hat{x}$, $d\hat{x} \in \hat{{\mathcal C}}^u_{\hat{x}}$
  and any $x$,  $dx \in {\mathcal C}^u_{x}$.
\end{lemma}

In the limit $\eta_{\mathrm{max}} \to 0$ and $\eta_{1,\mathrm{max}}\to 0$
the result of Lemma~\ref{lem_uCones_invariance} allows for the choice
of parameters
$\kappa = 0$,
$ {\mathcal V}_{\mathrm{min}} = {\mathcal K}_{\mathrm{min}} $,
$ {\mathcal V}_{\mathrm{max}} = {\mathcal K}_{\mathrm{max}} + \frac{1}{ \tau_{\mathrm{min}} } $.
In this case the minimal expansion rate
$\Lambda$, as defined in
Lemma~\ref{lem_uCones_expansion}, take on the form
$
\Lambda
=
1 + 2\,\tau_{\mathrm{min}}\,{\mathcal K}_{\mathrm{min}}
$.
These are the unstable cone field ${\mathcal C}^u$ for ${\mathcal F}$
and the corresponding expansion rate (with respect to $\left| \, \cdot \, \right|_*$)
as they are usually used in the theory of hyperbolic billiards
\cite{MR2229799}.

Throughout we will assume that
$\eta_{\mathrm{max}}$, $\eta_{1,\mathrm{max}}$, $\eta_{2,\mathrm{max}}$ and the parameters
${\mathcal V}_{\mathrm{min}}$, ${\mathcal V}_{\mathrm{max}}$, $\kappa$ are chosen
such that:
\begin{assumption}
  \label{assume_assumption1}
  The parameters $\kappa$, ${\mathcal V}_{\mathrm{min}}$, ${\mathcal V}_{\mathrm{max}}$
  chosen so that equality holds in Lemma~\ref{lem_uCones_invariance},
  and the cone fields $\hat{{\mathcal C}}^u$, ${\mathcal C}^u$ are invariant with
  $\Lambda > 1 + \tau_{\mathrm{min}} {\mathcal K}_{\mathrm{min}}$.
\end{assumption}

\begin{remark}
  As was already pointed out earlier,
  by Lemma~\ref{lem_uCones_invariance} and Lemma~\ref{lem_uCones_expansion}
  Assumption~\ref{assume_assumption1}
  can always be realized as long as
  $\eta_{\mathrm{max}}$, $\eta_{1,\mathrm{max}}$ are sufficiently small
  compared to geometric parameters of the billiard table $Q$.
  The smallness assumption on $\eta_{2,\mathrm{max}}$ will be imposed
  later on to ensure certain regularity of $\hat{{\mathcal F}}$.
\end{remark}

\section{Dynamics of unstable curves}
\label{sect_ucurves}

A curve $\hat{\gamma}$ in $\hat{{\mathcal M}}$ is called an unstable
curve if all its tangent vectors are in
the unstable cone $\hat{{\mathcal C}}^u_{\hat{\gamma}}$. This is in complete
analogy to the corresponding concept in the theory of
hyperbolic billiards \cite{MR2229799}. In fact, this is more
than a formal analogy since
Lemma~\ref{lem_uCones_invariance} and Lemma~\ref{lem_uCones_expansion}
show that the projection $\gamma = \Pi(\hat{\gamma})$ of
any unstable curve $\hat{\gamma}$ in $\hat{{\mathcal M}}$ is
an unstable curve in ${\mathcal M}$.
And since the dynamics of unstable curves is central in the
study for hyperbolic billiards \cite{MR2229799} we
derive the corresponding results for the dynamics of
unstable curves under iterations of $\hat{{\mathcal F}}$.

Let $\hat{\gamma}$ be some unstable curve in $\hat{{\mathcal M}}$.
By its very definition \eqref{eqn_def_uConeExt} we see
that $\hat{\gamma}$ can be parametrized in terms of $s$.
Furthermore, the assumed bound
$
| \frac{d\mathsf c}{ \mathsf c\,\cos\varphi\,ds } |
\leq
\kappa
$
on $\hat{\gamma}$ clearly implies the uniform bound
\begin{equation}
  \label{eqn_uCurveExt_variationC}
  \sup_{ \hat{x}', \hat{x}'' \in \hat{\gamma} }
  \frac{ \mathsf c' }{ \mathsf c'' }
  \leq
  e^{ \kappa \,( | \partial Q | \operatorname{\wedge} \frac{ \pi }{ {\mathcal V}_{\mathrm{min}} }) }
  =
  e^{ (\eta_{\mathrm{max}} + \eta_{1,\mathrm{max}} )\times {\mathrm{const}} }
\end{equation}
on the variation of $\mathsf c$ along $\hat{\gamma}$.
And since the cone field $\hat{{\mathcal C}}^u$ is invariant under $\hat{{\mathcal F}}$
it follows that images of unstable curves remain unstable curves
with the same uniform bounds, and projecting any of these unstable
curves by $\Pi$ from $\hat{{\mathcal M}}$ to ${\mathcal M}$ yields an
unstable curve for the standard billiard map ${\mathcal F}$ on ${\mathcal M}$.
\begin{figure}[ht!]
  \centering
  \includegraphics[height=5cm]{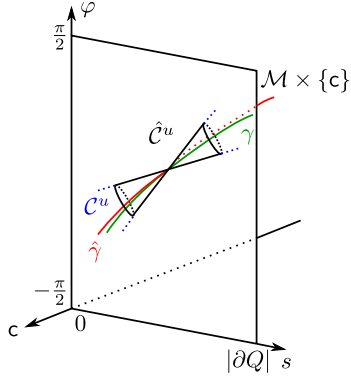}
  \caption{An illustration of the cone fields ${\mathcal C}^u$ and $\hat{{\mathcal C}}^u$
  as well as unstable curves $\hat{\gamma}$ and their projections
  $\gamma$ onto ${\mathcal M}$ (or in fact
  ${\mathcal M} \times \{\mathsf c\}$). The cones $\hat{{\mathcal C}}^u$
  has a narrow opening in the $\mathsf c$--coordinate, and
  hence unstable curves $\hat{\gamma}$ almost agree with their
  projections $\gamma$.}
  \label{fig_cones}
\end{figure}
An illustration of the unstable cone field and unstable curves, as
well as their projections is given in Fig.~\ref{fig_cones}.
As our standing assumption is $\eta_{2,\mathrm{max}} < \infty$ (in fact
it will be eventually assumed to very very small) we have
bounded second derivatives of $P$. Thus
the curvature of $\hat{{\mathcal F}}(\hat{\gamma})$ is bounded as long as
the curvature of $\hat{\gamma}$ is bounded. For planar
hyperbolic billiards \cite{MR2197961} shows that
the curvature of unstable curves is uniformly bounded
under iteration by the standard billiard map.
The map $\hat{{\mathcal F}}$ we consider here is essentially a small
perturbation of the standard billiard ${\mathcal F}$. Indeed,
a straightforward
adaptation of the proofs in \cite{MR2197961}, as was done
also in similar settings \cite{MR2389891,MR2737493}, we obtain
the following:
\begin{lemma}[Uniform curvature bounds]
  \label{lem_curvature_bounds}
  Suppose that the curvature of an unstable curve $\hat{\gamma}$
  is bounded by some constant $C_0$. Then uniformly in $n \geq 1$
  the curvature of $\hat{{\mathcal F}}^n(\hat{\gamma})$ is bounded by
  some constant
  $
  C
  $, which depends on $C_0$, $Q$, and is independent of
  $\eta_{\mathrm{max}}$, $\eta_{1,\mathrm{max}}$, $\eta_{2,\mathrm{max}}$
  provided that they are chosen less than some
  $\eta_{\mathrm{max}}^*$, $\eta_{1,\mathrm{max}}^*$, $\eta_{2,\mathrm{max}}^*$,
  respectively.
\end{lemma}

Due to Lemma~\ref{lem_curvature_bounds} we will throughout make the
following assumption:
\begin{assumption}
  \label{assume_curvatureBounds}
  All unstable curves $\hat{\gamma}$ considered have a universally
  bounded curvature, uniformly in the parameters
  $\eta_{\mathrm{max}}\leq\eta_{\mathrm{max}}^*$,
  $\eta_{1,\mathrm{max}}\leq\eta_{1,\mathrm{max}}^*$,
  $\eta_{2,\mathrm{max}} \leq\eta_{2,\mathrm{max}}^*$.
\end{assumption}

In order to control distortions of unstable curves under
iterations by $\hat{{\mathcal F}}$ we make again use of the fact that
$\hat{{\mathcal F}}$ is a perturbation of the standard billiard map ${\mathcal F}$,
whose distortion estimates are well-understood, e.g.
\cite{MR2229799}.
The natural generalization of the so-called the homogeneity
strips used in hyperbolic billiards \cite{MR1138952,MR1071936} are
\begin{equation}
  \label{eqn_def_HstripExt}
  \begin{split}
    \forall\; k \geq k_0
    \;:
    \quad
    \hat{{\mathbb H}}_k
    &=
    \Big\{
    (s, \varphi, \mathsf c)
    \operatorname{:}
    \frac{\pi}{2} - \frac{1}{k^2}
    <
    \varphi
    <
    \frac{\pi}{2} - \frac{1}{(k+1)^2}
    \Big\}
    \\
    \hat{{\mathbb H}}_0
    &=
    \Big\{
    (s, \varphi, \mathsf c)
    \operatorname{:}
    -\frac{\pi}{2} + \frac{1}{k_0^2}
    <
    \varphi
    <
    \frac{\pi}{2} - \frac{1}{k_0^2}
    \Big\}
    \\
    \forall\; k \geq k_0
    \;:
    \quad
    \hat{{\mathbb H}}_{-k}
    &=
    \Big\{
    (s, \varphi, \mathsf c)
    \operatorname{:}
    -\frac{\pi}{2} + \frac{1}{(k+1)^2}
    <
    \varphi
    <
    -\frac{\pi}{2} + \frac{1}{k^2}
    \Big\}
    \;,
  \end{split}
\end{equation}
which we will call homogeneity surfaces.
The value of $k_0 \geq 1$ is determined by the one-step expansion
property \cite{MR2229799} stated in Lemma~\ref{lem_oneStep_expansion}
below.

In order to make use of the homogeneity surfaces
when estimating distortions of images of unstable
curves under iterations of $\hat{{\mathcal F}}$ it is convenient
to follow standard practice of hyperbolic billiards
and introduce additional
singularities for $\hat{{\mathcal F}}$ (i.e. artificial singularities
in addition to the ones present in $\hat{{\mathcal F}}$ due to
${\mathcal F}$) as follows:
\begin{itemize}
  \item
    An unstable curve $\hat{\gamma}$ in $\hat{{\mathcal M}}$
    which does not cross any of the homogeneity surfaces
    $(\hat{{\mathbb H}}_k)_k$ is called a weakly homogeneous unstable curve.
  \item
    The surfaces $(\hat{{\mathbb H}}_k)_k$ act as additional
    singularities of $\hat{{\mathcal F}}$, hence if any of the
    image under $\hat{{\mathcal F}}$
    of any weakly homogeneous unstable curve $\hat{\gamma}$
    crossing any of the
    $(\hat{{\mathbb H}}_k)_k$ will be cut accordingly into
    weakly homogeneous unstable curves.
\end{itemize}
In particular, the image under $\hat{{\mathcal F}}$ of any
weakly homogeneous unstable curve
is a finite or countable union of
weakly homogeneous unstable curves.

For any (weakly homogeneous) unstable curves
$\hat{\gamma}_0$, $\hat{\gamma}_1$
with $\hat{{\mathcal F}}(\hat{\gamma}_0) = \hat{\gamma}_1$
and any $\hat{x}_0 \in \hat{\gamma}_0$
we denote by ${\mathcal J}_{\hat{\gamma}_0}\hat{{\mathcal F}}(\hat{x}_0)$
the Jacobian of $\hat{{\mathcal F}} \colon \hat{\gamma}_0 \to \hat{\gamma}_1$
at $\hat{x}_0$. Similarly, we denote by
${\mathcal J}_{\hat{\gamma}_1}\hat{{\mathcal F}}^{-1}(\hat{x}_1)$
the Jacobian of $\hat{{\mathcal F}}^{-1}$.
Since we already know that unstable curves $\hat{\gamma}$
in ${\mathcal M}$ are very close to their projections
$\gamma = \Pi \hat{\gamma}$ in ${\mathcal M}$, the following
uniform distortion bound follows from the corresponding standard
arguments for hyperbolic billiards
\cite{MR2229799,MR2737493}:
\begin{lemma}[Uniform distortion bounds]
  \label{lem_distortion_bounds}
  For any choice of
  $\eta_{\mathrm{max}}^*$, $\eta_{1,\mathrm{max}}^*$, $\eta_{2,\mathrm{max}}^*$
  there exists a constant $C$ such that for every
  weakly homogeneous unstable curves
  $ \hat{\gamma}_0 $ and $ \hat{\gamma}_1 = \hat{{\mathcal F}}( \hat{\gamma}_0 ) $
  \begin{equation*}
    \sup_{ \hat{x}_1 \in \hat{\gamma}_1 }
    \Big|
    \frac{d}{ds_1}
    \log {\mathcal J}_{\hat{\gamma}_1} {\mathcal F}^{-1}(\hat{x}_1)
    \Big|
    \leq
    \frac{C}{|\hat{\gamma}_1|^{\frac{2}{3}} }
  \end{equation*}
  holds, uniformly for in
  $\eta_{\mathrm{max}}$, $\eta_{1,\mathrm{max}}$, $\eta_{2,\mathrm{max}}$
  provided that they are chosen less than
  $\eta_{\mathrm{max}}^*$, $\eta_{1,\mathrm{max}}^*$, $\eta_{2,\mathrm{max}}^*$,
  respectively.
\end{lemma}

The key result in the study of hyperbolic billiards is the
so-called one-step expansion \cite{MR2229799} property. Since
unstable curves $\hat{\gamma}$ are uniformly close to their
projections $\gamma = \Pi \hat{\gamma}$ also this property of
${\mathcal F}$ readily carries over to our setting of $\hat{{\mathcal F}}$.
To be precise, let $\hat{\gamma}$ be a weakly homogeneous
unstable curve.
Recall that $\hat{{\mathcal F}}(\hat{\gamma})$ is cut into
several connected component due to the presence of singularities
in ${\mathcal F}$, and due to the additional singularities introduced by
the homogeneity surfaces. For any connected component
$\hat{\gamma}_i$ of $\hat{{\mathcal F}}(\hat{\gamma})$ denote by
$\lambda_i$ the minimal expansion of $\hat{{\mathcal F}}$
on $\hat{{\mathcal F}}^{-1} \hat{\gamma}_i$ in terms of the adapted
metric $\left| \, \cdot \, \right|_*$.
\begin{lemma}[Uniform one-step expansion]
  \label{lem_oneStep_expansion}
  For any $\eta_{\mathrm{max}}^*$, $\eta_{1,\mathrm{max}}^*$, $\eta_{2,\mathrm{max}}^*$
  \begin{equation*}
    \liminf_{l \to 0}
    \sup_{ \eta_{\mathrm{max}}, \eta_{1,\mathrm{max}}, \eta_{2,\mathrm{max}} }
    \sup_{\hat{\gamma} \operatorname{:} |\hat{\gamma}| < l }
    \sum_i \lambda_i
    <
    1
  \end{equation*}
  where the supremum is taken over all weakly unstable curves
  $\hat{\gamma}$ and all
  $\eta_{\mathrm{max}}$, $\eta_{1,\mathrm{max}}$, $\eta_{2,\mathrm{max}}$
  less than
  $\eta_{\mathrm{max}}^*$, $\eta_{1,\mathrm{max}}^*$, $\eta_{2,\mathrm{max}}^*$,
  respectively.
\end{lemma}

Once a one-step expansion such as in Lemma~\ref{lem_oneStep_expansion}
is established the so-called growth lemma follow from general
arguments as explained in \cite{MR2229799,MR1832968,MR2737493}
and references therein.
In order to formulate it we introduce the following notations.
For any weakly homogeneous unstable curve $\hat{\gamma}$
we denote by $ \mathrm{m}_{\hat{\gamma}} $ the Lebesgue measure
on it. For every $n \geq0$ its image $\hat{{\mathcal F}}(\hat{\gamma})$
consists of a finite or countable number of weakly homogeneous
unstable curves, and for every $\hat{x} \in \hat{\gamma}$
we denote by $\hat{\gamma}_n(\hat{x})$ the component of
$\hat{{\mathcal F}}^n(\hat{\gamma})$ containing $\hat{{\mathcal F}}^n(\hat{x})$.
Furthermore, we denote by
\begin{equation*}
  r_{n}(\hat{x})
  =
  \mathrm{dist}_{ \hat{\gamma}_n(\hat{x}) }(
  \hat{{\mathcal F}}^n(\hat{x}), \partial \hat{\gamma}_n(\hat{x})
  )
\end{equation*}
the distance of the point $\hat{{\mathcal F}}^n(\hat{x})$ to the
closest endpoints of the component of
$\hat{{\mathcal F}}^n(\hat{\gamma})$ containing it.
With this notation in place we can formulate the
aforementioned growth lemma for $\hat{{\mathcal F}}$, whose
proof can be found in \cite{MR2229799,MR1832968,MR2737493},
where the particular formulation given below can be
found in \cite{MR2737493}.

\begin{lemma}[Uniform growth lemma]
  \label{lem_growth_lemma}
  Fix $\eta_{\mathrm{max}}^*$, $\eta_{1,\mathrm{max}}^*$, $\eta_{2,\mathrm{max}}^*$.
  Then uniformly in
  $\eta_{\mathrm{max}}$, $\eta_{1,\mathrm{max}}$, $\eta_{2,\mathrm{max}}$
  less than
  $\eta_{\mathrm{max}}^*$, $\eta_{1,\mathrm{max}}^*$, $\eta_{2,\mathrm{max}}^*$,
  respectively,
  and uniformly for any weakly homogeneous unstable curve
  $\hat{\gamma}$ the following hold:
  \begin{enumerate}[(a)]
    \item
      \label{item_lem_growth_lemma_a}
      There exists $0 < \theta_0 <1$, $c_1, c_2 > 0$ such that
      \begin{equation*}
         \mathrm{m}_{\hat{\gamma}} \{ r_{n} < \zeta \}
        \leq
        c_1\,(\theta_0\,\Lambda)^n
        \,  \mathrm{m}_{\hat{\gamma}} \{
        r_{0} < \zeta\,\Lambda^{-n} \}
        +
        c_2\, \zeta\, |\hat{\gamma}|
      \end{equation*}
      holds for all $n\geq 0$ and all $\zeta>0$.

    \item
      \label{item_lem_growth_lemma_b}
      There exist $c_3, c_4 >0$ such that whenever
      $n \geq c_3\,| \log|\hat{\gamma}||$, then
      \begin{equation*}
         \mathrm{m}_{\hat{\gamma}} \{ r_{n} < \zeta \}
        \leq
        c_4\, \zeta\, |\hat{\gamma}|
      \end{equation*}
      for any $\zeta>0$.

    \item
      \label{item_lem_growth_lemma_c}
      There exist $0 < \theta_1 < 1$, $c_5, c_6>0$, $\zeta_0>0$
      such that
      \begin{equation*}
         \mathrm{m}_{\hat{\gamma}} \{
        \max_{n \operatorname{:} n_1 < n < n_2} r_{n}
        < \zeta_0 \}
        \leq
        c_6\,\theta_1^{n_2-n_1}\,| \hat{\gamma}|
      \end{equation*}
      holds for all $n_2> n_1>c_5\,| \log |\hat{\gamma}| |$.
  \end{enumerate}
\end{lemma}

In the theory of hyperbolic billiards the growth lemma is
the key tool to derive strong statistical properties of
the billiard map ${\mathcal F}$ via standard pairs (see below).
In the following we state the relevant results for $\hat{{\mathcal F}}$
that follow from the corresponding results for ${\mathcal F}$ with
only minor changes in their proofs.
We refer to \cite{MR2229799,MR1832968,MR2737493}
and references therein for detailed proofs.

For any two points $\hat{x}', \hat{x}'' \in \hat{{\mathcal M}}$
we denote by $s_+(\hat{x}', \hat{x}'')$
the smallest $n\geq 0$ for which the corresponding image points
$\hat{{\mathcal F}}^(\hat{x}')$ and $\hat{{\mathcal F}}^n(\hat{x}'')$
are separated by either a singularity surface of $\hat{{\mathcal F}}$
or a homogeneity surface.
A standard pair is a weakly unstable curve $\hat{\gamma}$
with an absolutely continuous probability measure
$\hat{\rho}(\hat{x})\, \mathrm{m}_{\hat{\gamma}} (d\hat{x})$
on it, whose density satisfies
\begin{equation}
  \label{eqn_def_standardPair_density}
  |
  \log
  \hat{\rho}(\hat{x}')
  -
  \log
  \hat{\rho}(\hat{x}'')
  |
  \leq
  C_*\,\Lambda^{ - s_+(\hat{x}', \hat{x}'') }
  \quad \text{for all}\quad
  \hat{x}', \hat{x}'' \in \hat{{\mathcal M}}
  \;,
\end{equation}
where the (sufficiently large) constant $C_*$ is independent
of $\eta_{\mathrm{max}}$, $\eta_{1,\mathrm{max}}$, $\eta_{2,\mathrm{max}}$,
provided that
$\eta_{\mathrm{max}}$, $\eta_{1,\mathrm{max}}$, $\eta_{2,\mathrm{max}}$
less than
$\eta_{\mathrm{max}}^*$, $\eta_{1,\mathrm{max}}^*$, $\eta_{2,\mathrm{max}}^*$,
for some
$\eta_{\mathrm{max}}^*$, $\eta_{1,\mathrm{max}}^*$, $\eta_{2,\mathrm{max}}^*$.

Due to the distortion bound
Lemma~\ref{lem_distortion_bounds}
the image under $\hat{{\mathcal F}}^n$ of any standard pair 
is the union of finitely or countably many standard pairs.
Generalizing to linear combinations of standard pairs,
we say \cite{MR2229799} that a (possibly uncountable) collection
$(\hat{\gamma}_\alpha, \hat{\rho}_\alpha)_{
\alpha \in {\mathcal A}}$
of standard pairs with measure $\lambda(d\alpha)$ on
${\mathcal A}$ forms a standard family, which we will usually
denote by $\hat{{\mathcal G}}$.
For any Borel set $B \subset \hat{{\mathcal M}}$
denote by $\nu_{\hat{{\mathcal G}}}(B)$
\begin{equation}
  \label{eqn_def_measureG}
  \nu_{\hat{{\mathcal G}}}(B)
  =
  \int_{{\mathcal A}}
  \int_{B \cap \hat{\gamma}_\alpha}
  \hat{\rho}_\alpha(\hat{x})
  \, \mathrm{m}_{\hat{\gamma}_\alpha} (d\hat{x})
  \,\lambda(d\alpha)
\end{equation}
the corresponding probability measure on $\hat{{\mathcal M}}$.
The crucial observation is that the image under $\hat{{\mathcal F}}^n$
of any standard family is again a standard family.

Following standard terminology \cite{MR2229799} we introduce
the following concepts for a given standard family
$\hat{{\mathcal G}}$. Any $\hat{x} \in \hat{\gamma}_\alpha$ divides
$\hat{\gamma}_\alpha$ into two parts, and we denote by
$r_{\hat{{\mathcal G}}}(\hat{x})$ the length of the shorter one.
Correspondingly we introduce
\begin{equation*}
  {\mathcal Z}_{\hat{{\mathcal G}}}
  =
  \sup_{\zeta > 0} \frac{1}{\zeta}\,\nu_{\hat{{\mathcal G}}}\{
  r_{\hat{{\mathcal G}}} < \zeta \}
  \;,
\end{equation*}
which measures the typical length of curves in $\hat{{\mathcal G}}$.
Indeed
\begin{equation}
  \label{eqn_Z_asymp}
  {\mathcal Z}_{\hat{{\mathcal G}}}
  \asymp
  \int_{{\mathcal A}}
  \frac{ \lambda( d\alpha ) }{ | \hat{\gamma}_{\alpha} | }
\end{equation}
The growth lemma Lemma~\ref{lem_growth_lemma} implies
that for a standard family $\hat{{\mathcal G}}$ with
${\mathcal Z}_{\hat{{\mathcal G}}}<\infty$
\begin{equation}
  \label{eqn_StandardFamilyGrowth}
  \hat{{\mathcal G}}_n = \hat{{\mathcal F}}^n(\hat{{\mathcal G}})
  \;,\quad
  {\mathcal Z}_{\hat{{\mathcal G}}_n} \leq C\,( \theta^n\, {\mathcal Z}_{\hat{{\mathcal G}}} + 1 )
\end{equation}
for some $0 < \theta < 1$ and all $n\geq 1$.

We say \cite{MR2229799} that a standard pair
$(\hat{\gamma}, \hat{\rho})$
is a proper standard pair if $|\hat{\gamma}| \geq \ell_p$,
where
\begin{equation}
  \ell_p > 0
\end{equation}
is a (small, but) fixed constant.
We say that a standard family $\hat{{\mathcal G}}$ is a proper
standard family if ${\mathcal Z}_{\hat{{\mathcal G}}} < Z_p$, where
\begin{equation}
  \label{eqn_def_Zmin}
  Z_p > 0
\end{equation}
is a (large, but) fixed constant, which is chosen (in
relation to $\ell_p$) such that
all standard pairs are proper standard families.
Moreover:
\begin{lemma}[Invariance of standard families]
  \label{lem_invariance_properStandardFamilies}
  For every $n\geq 0$, the image under $\hat{{\mathcal F}}^n$
  of any proper standard family
  is again a proper standard family.
\end{lemma}

And as direct consequence of the fact that unstable curves
$\hat{\gamma}$ are uniformly close to their projections
$\Pi\hat{\gamma}$ we obtain
\begin{lemma}[Projections of proper standard families]
  \label{lem_projection_pG}
  For any proper standard family $\hat{{\mathcal G}}$ for $\hat{{\mathcal F}}$
  its projection ${\mathcal G} = \Pi \hat{{\mathcal G}}$
  is a proper standard family for ${\mathcal F}$.
\end{lemma}
\begin{lemma}[Lifting of proper standard families]
  \label{lem_lift_pG}
  For any proper standard family ${\mathcal G}$ for ${\mathcal F}$
  its lift $\hat{{\mathcal G}} = {\mathcal G} \times \{\mathsf c\}$
  is a proper standard family for $\hat{{\mathcal F}}$ for any
  $\mathsf c$.
\end{lemma}

We finish this section pointing out that not all tools
used in the study of hyperbolic billiards carry over to
$\hat{{\mathcal F}}$. Namely, the above mentioned results
are essentially due to the fact that unstable curves in
$\hat{{\mathcal M}}$ are very close to their projections to
${\mathcal M}$, and hence their dynamics are comparable.
On the other hand, the so-called coupling lemma, which is used
to derive statistical properties of hyperbolic billiards
\cite{MR2229799} also requires recurrence, which clearly
is not given for the dynamics of $\hat{{\mathcal F}}$, as the dynamics
in $\mathsf c$--coordinates prevents recurrence in
$\hat{{\mathcal M}}$. However, we shall show that the
almost recurrence in the projection onto ${\mathcal M}$
turns out to be useful to effectively approximate
the dynamics under $\hat{{\mathcal F}}$ on $\hat{{\mathcal M}}$.

\section{Billiard approximation}

In the limit as $\eta_{\mathrm{max}} \to 0$ the map $\hat{{\mathcal F}}$
converges to the map $\hat{{\mathcal F}}_0$,
defined in \eqref{eqn_def_bMapExtZero},
which is the usual billiard map ${\mathcal F}$ in the
$(s, \varphi)$--coordinates
combined with the identity map in the $\mathsf c$--coordinate.
As outlined in the introduction, we are interested in describing
the dynamics of the $\mathsf c$--coordinate under $\hat{{\mathcal F}}$
in precisely this limiting regime where $\eta_{\mathrm{max}}$ is small.

Observe that if
$ \hat{x}_1 = \hat{{\mathcal F}}(\hat{x}_0) $,
then we conclude from \eqref{eqn_updateRule_phi_c_explicit}
\begin{equation}
  \label{eqn_speed_apriori}
  1 - \eta_{\mathrm{max}}
  \leq
  \frac{ \mathsf c_1 }{ \mathsf c_0 }
  \leq
  1
  \;.
\end{equation}
Therefore, as $\eta \to 0$ we naturally have a slow-fast system,
where the fast coordinates $(s, \varphi)$ essentially evolve
according to the billiard map ${\mathcal F}$. And since ${\mathcal F}$ is
known \cite{MR2229799} to have strong statistical properties,
we expect an averaging method to allow us to derive a closed
equation for $\mathsf c$ on a time-scale on which
$\mathsf c$ changes of order one, i.e. for a number of iterates
of $\hat{{\mathcal F}}$ of order $\operatorname{\mathcal{O}}(\frac{1}{\eta_{\mathrm{max}}} )$.

See \cite{MR2062922,MR2547839,MR2241812}
for related results on averaging in fully coupled smooth systems.
In the recent work \cite{MR3640023} averaging result for
non-smooth systems are derived, however, these results do not
cover the fully coupled setting.

Our strategy to derive an averaging result for $\mathsf c$
is to employ the methods of \cite{MR2499824,MR2034323,MR2031432,MR2241812}
so that we can handle the singularities of the map $\hat{{\mathcal F}}$
on time-scales of order $\operatorname{\mathcal{O}}(\frac{1}{\eta_{\mathrm{max}}} )$.

To shorten a subscript $n$ on $\hat{x}$, i.e.
$\hat{x}_n$, will always signify an orbit
under $\hat{{\mathcal F}}$. Whenever clear from the context, 
given $\hat{x}_n$, we let $x_n$ and
$\mathsf c_n \equiv (s_n, \varphi_n)$
denote the corresponding $(s, \varphi)$--component
and $\mathsf c$--component, respectively.
Furthermore we introduce the function $g$
on $\hat{{\mathcal M}}$ as
\begin{equation}
  \label{eqn_def_slowFast_bFlowSpeedInc}
  \begin{split}
    g(s, \tilde{\varphi}, \mathsf c)
    &=
    \mathsf c
    \,\Big[
    \sqrt{
    [ 1 - \eta(\mathsf c\,\cos\tilde{\varphi} ) ]^2
    \,\cos^2\tilde{\varphi}
    +
    \sin^2\tilde{\varphi}
    }
    -
    1
    \Big]
    \\
    &=
    -
    \mathsf c
    \,\frac{
    [ 2 - \eta(\mathsf c\,\cos\tilde{\varphi} ) ]
    \,\eta(\mathsf c\,\cos\tilde{\varphi} )
    \,\cos^2\tilde{\varphi}
    }{
    1
    +
    \sqrt{
    [ 1 - \eta(\mathsf c\,\cos\tilde{\varphi} ) ]^2
    \,\cos^2\tilde{\varphi}
    +
    \sin^2\tilde{\varphi}
    }
    }
  \end{split}
\end{equation}
so that the evolution of $\mathsf c$ as stated in
\eqref{eqn_updateRule_phi_c_explicit}
takes on the form
\begin{equation}
  \label{eqn_speed_oneStep}
  \mathsf c_{n+1}
  -
  \mathsf c_n
  =
  g(\hat{{\mathcal F}}_0(\hat{x}_n))
  \equiv
  g({\mathcal F}(s_n, \varphi_n), \mathsf c_n)
  \;.
\end{equation}
Iterating \eqref{eqn_speed_oneStep} we obtain for any $m\geq 1$
and any $n$
\begin{equation*}
  \mathsf c_{n+m}
  -
  \mathsf c_n
  =
  \sum_{k=n}^{n+m-1}
  g(\hat{{\mathcal F}}_0(\hat{x}_k))
  \equiv
  \sum_{k=n}^{n+m-1}
  g({\mathcal F}(x_k), \mathsf c_k)
  \;.
\end{equation*}
For later use we record the following elementary estimates
on $g$
\begin{equation}
  \label{eqn_estimates_bFlowSpeedInc}
  \begin{split}
    -
    \eta_{\mathrm{max}} \,\mathsf c
    &\leq
    g(s, \tilde{\varphi}, \mathsf c)
    \leq
    0
    \;,\quad
    |
    \partial_{\mathsf c}
    g(s, \tilde{\varphi}, \mathsf c)
    |
    \leq
    \eta_{\mathrm{max}} + \eta_{1,\mathrm{max}}
    \\
    \frac{\partial}{\partial s}
    g(s, \tilde{\varphi}, \mathsf c)
    &=
    0
    \;,\quad
    \Big|
    \frac{\partial}{\partial\tilde{\varphi}}
    g(s, \tilde{\varphi}, \mathsf c)
    \Big|
    \leq
    \mathsf c
    \,\Big[
    \frac{ ( 2 - \eta_{\mathrm{max}} )\,\eta_{\mathrm{max}} }{ 1 - \eta_{\mathrm{max}} }
    +
    \eta_{1,\mathrm{max}}
    \Big]
    \;,
  \end{split}
\end{equation}
which hold uniformly for all
$(s, \tilde{\varphi}, \mathsf c)$.

In the following we will let $A \colon {\mathbb R}\to{\mathbb R}$ denote
a $C^2$ bounded function with bounded
first and second derivatives.
With the above expression for $\mathsf c_{n+m} - \mathsf c_n$
it follows that for any $n$ and $m\geq 1$
\begin{equation*}
  A(\mathsf c_{n+m}) - A(\mathsf c_n)
  =
  A'( \mathsf c_n )
  \sum_{k=n}^{n+m-1}
  g({\mathcal F}(x_k), \mathsf c_k)
  +
  R_{n,m}^{(1)}
\end{equation*}
where
\begin{equation*}
  | R_{n,m}^{(1)} |
  \leq
  \frac{1}{2}
  \,\Big|
  \sum_{k=n}^{n+m-1}
  g({\mathcal F}(x_k), \mathsf c_k)
  \Big|^2
  \,\left| \,  A''  \, \right|_\infty
  \leq
  \eta_{\mathrm{max}}^2 \, \frac{1}{2} \, m^2
  \,\mathsf c_n^2
  \,\left| \,  A''  \, \right|_\infty
\end{equation*}
follows from \eqref{eqn_estimates_bFlowSpeedInc}.
It also follows readily from \eqref{eqn_estimates_bFlowSpeedInc}
that
\begin{equation*}
  \Big|
  \sum_{k=n}^{n+m-1}
  g({\mathcal F}(x_k), \mathsf c_k)
  -
  \sum_{k=n}^{n+m-1}
  g({\mathcal F}(x_k), \mathsf c_n)
  \Big|
  \leq
  \eta_{\mathrm{max}}\,( \eta_{\mathrm{max}} + \eta_{1,\mathrm{max}} )
  \,\frac{1}{2} \,(m-1)\,m
  \,\mathsf c_n
  \;,
\end{equation*}
and hence
\begin{equation}
  \label{eqn_est_dA_1}
  A(\mathsf c_{n+m}) - A(\mathsf c_n)
  =
  A'( \mathsf c_n )
  \sum_{k=0}^{m-1}
  g({\mathcal F}\circ \Pi \circ \hat{{\mathcal F}}^k(\hat{x}_n),
  \mathsf c_n)
  +
  R_{n,m}^{(1)}
  +
  R_{n,m}^{(2)}
\end{equation}
with
\begin{equation*}
  | R_{n,m}^{(2)} |
  \leq
  \eta_{\mathrm{max}}\,( \eta_{\mathrm{max}} + \eta_{1,\mathrm{max}} )
  \,\frac{1}{2} \,(m-1)\,m
  \,\mathsf c_n
  \,\left| \,  A'  \, \right|_\infty
  \;.
\end{equation*}
We want to stress that the estimates on
$| R_{n,m}^{(1)} |$ are uniform in $( \hat{x}_k )_{k \geq n}$
$| R_{n,m}^{(2)} |$.

Suppose that the distribution of $\hat{x}_0$
is given by a proper standard family
${\mathcal G}_0$. The invariance property
Lemma~\ref{lem_invariance_properStandardFamilies}
implies that the distribution of $\hat{x}_n$
is given by the proper standard family
$ {\mathcal G}_n = \hat{{\mathcal F}}^n {\mathcal G}_0 $.
\begin{lemma}
  \label{lem_est_dA_2}
  For any $\eta_{\mathrm{max}}^*$, $\eta_{1,\mathrm{max}}^*$, $\eta_{2,\mathrm{max}}^*$,
  there exists a constant $C>0$ such that
  \begin{align*}
    \int
    &
    [ A(\mathsf c_{n+m}) - A(\mathsf c_n) ]
    \,\nu_{ {\mathcal G}_0 }(d\hat{x}_0)
    =
    R_{n,m}
    +
    \\
    &+
    \sum_{k=0}^{m-1}
    \int_{{\mathcal A}_n}
    \int_{\hat{\gamma}_{n,\alpha}}
    A'( \mathsf c_{n,\alpha} )
    \,g({\mathcal F}\circ \Pi \circ \hat{{\mathcal F}}^k(\hat{x}),
    \mathsf c_{n,\alpha} )
    \,\hat{\rho}_{n,\alpha}(\hat{x})
    \, \mathrm{m}_{\hat{\gamma}_{n,\alpha}} (d\hat{x})
    \,\lambda_n(d\alpha)
    \;,
  \end{align*}
  where each
  $ \mathsf c_{n,\alpha} $ denotes an arbitrary
  $\mathsf c$--value on $\hat{\gamma}_{n,\alpha}$
  (e.g. the average value
  $\mathsf c_{(\hat{\gamma}_{n,\alpha},\hat{\rho}_{n,\alpha})}$
  of $\mathsf c$ along $\hat{\gamma}_{n,\alpha}$),
  and
  \begin{equation*}
    | R_{n,m} |
    \leq
    (\eta_{\mathrm{max}} + \eta_{1,\mathrm{max}} )^2
    \,m \,\Big( \frac{1}{2} \,m + C \Big)
    \int
    \Big[ \mathsf c_n \,\left| \,  A''  \, \right|_\infty + \left| \,  A'  \, \right|_\infty \Big]
    \,\mathsf c_n
    \,\nu_{ {\mathcal G}_0 }(d\hat{x}_0)
  \end{equation*}
  uniformly in
  $\eta_{\mathrm{max}}$, $\eta_{1,\mathrm{max}}$, $\eta_{2,\mathrm{max}}$
  provided that they are chosen less than
  $\eta_{\mathrm{max}}^*$, $\eta_{1,\mathrm{max}}^*$, $\eta_{2,\mathrm{max}}^*$,
  respectively.
\end{lemma}
\begin{proof}
  From the preceding discussion
  leading to \eqref{eqn_est_dA_1}
  we obtain
  \begin{align*}
    \int
    [ A(\mathsf c_{n+m}) - A(\mathsf c_n) ]
    \,\nu_{ {\mathcal G}_0 }(d\hat{x}_0)
    &=
    \int
    \sum_{k=0}^{m-1}
    A'( \mathsf c )
    \,g({\mathcal F}\circ \Pi \circ \hat{{\mathcal F}}^k(\hat{x}),
    \mathsf c)
    \,\nu_{ {\mathcal G}_n }(d\hat{x})
    \\
    &\quad
    +
    \int
    ( R_{n,m}^{(1)} + R_{n,m}^{(2)} )
    \,\nu_{ {\mathcal G}_0 }(d\hat{x}_0)
  \end{align*}
  with
  \begin{align*}
    \Big|
    \int
    ( R_{n,m}^{(1)} + R_{n,m}^{(2)} )
    \,\nu_{ {\mathcal G}_0 }(d\hat{x}_0)
    \Big|
    &\leq
    \frac{1}{2}\,( \eta_{\mathrm{max}} + \eta_{1,\mathrm{max}} )^2 \,m^2
    \times
    \\
    &\quad
    \times
    \int
    \Big(
    \mathsf c_n
    \,\left| \,  A''  \, \right|_\infty
    +
    \,\left| \,  A'  \, \right|_\infty
    \Big)
    \,\mathsf c_n
    \,\nu_{ {\mathcal G}_0 }(d\hat{x}_0)
    \;.
  \end{align*}
  Writing the average with respect to
  $\nu_{ {\mathcal G}_n }$
  in terms of the individual standard pairs of the standard
  family, recall \eqref{eqn_def_measureG}, we have
  \begin{align*}
    \int
    &
    A'( \mathsf c )
    \,g({\mathcal F}\circ \Pi \circ \hat{{\mathcal F}}^k(\hat{x}),
    \mathsf c)
    \,\nu_{ {\mathcal G}_n }(d\hat{x})
    \\
    &=
    \int_{{\mathcal A}_n}
    \int_{\hat{\gamma}_{n,\alpha}}
    A'( \mathsf c )
    \,g({\mathcal F}\circ \Pi \circ \hat{{\mathcal F}}^k(\hat{x}),
    \mathsf c)
    \,\hat{\rho}_{n,\alpha}(\hat{x})
    \, \mathrm{m}_{\hat{\gamma}_{n,\alpha}} (d\hat{x})
    \,\lambda_n(d\alpha)
    \;.
  \end{align*}

  By \eqref{eqn_uCurveExt_variationC},
  the variation of $\mathsf c$ along any $\hat{\gamma}_{n,\alpha}$
  is small. In particular, there exists a constant $C>0$
  such that
  \begin{equation*}
    | \mathsf c_{n,\alpha} - \mathsf c |
    \leq
    (\eta_{\mathrm{max}} + \eta_{1,\mathrm{max}} )\,C \,\mathsf c
    \quad\text{for all}\quad
    \hat{x} \in \hat{\gamma}
    \;.
  \end{equation*}
  With this a straightforward argument similar to
  the above derivation of the estimate for
  $
  |
  \int
  ( R_{n,m}^{(1)} + R_{n,m}^{(2)} )
  \,\nu_{ {\mathcal G}_0 }(d\hat{x}_0)
  |
  $
  yields
  \begin{align*}
    \int
    &
    A'( \mathsf c )
    \,g({\mathcal F}\circ \Pi \circ \hat{{\mathcal F}}^k(\hat{x}),
    \mathsf c)
    \,\nu_{ {\mathcal G}_n }(d\hat{x})
    \\
    &=
    \int_{{\mathcal A}_n}
    \int_{\hat{\gamma}_{n,\alpha}}
    A'( \mathsf c_{n,\alpha} )
    \,g({\mathcal F}\circ \Pi \circ \hat{{\mathcal F}}^k(\hat{x}),
    \mathsf c_{n,\alpha} )
    \,\hat{\rho}_{n,\alpha}(\hat{x})
    \, \mathrm{m}_{\hat{\gamma}_{n,\alpha}} (d\hat{x})
    \,\lambda_n(d\alpha)
    +
    Q_1
    +
    Q_2
  \end{align*}
  where
  \begin{align*}
    |Q_1|
    &\leq
    (\eta_{\mathrm{max}} + \eta_{1,\mathrm{max}} )^2
    \,C
    \,\left| \,  A'  \, \right|_\infty
    \int \mathsf c_n \,\nu_{ {\mathcal G}_0 }(d\hat{x}_0)
    \\
    |Q_2|
    &\leq
    \eta_{\mathrm{max}} \,(\eta_{\mathrm{max}} + \eta_{1,\mathrm{max}} )
    \,C
    \left| \,  A''  \, \right|_\infty
    \int \mathsf c_n^2 \,\nu_{ {\mathcal G}_0 }(d\hat{x}_0)
    \;.
  \end{align*}
  Combining these yields the claimed estimate.
\end{proof}

At this point we would like to comment on an important aspect of the
result of Lemma~\ref{lem_est_dA_2}. Namely, replacing
$\mathsf c$ by some $\mathsf c_{n,\alpha}$
in each of the integrals along $\hat{\gamma}_{n,\alpha}$
turns out to be essential in order to proceed with the
analysis of these integrals. This is because although
$\mathsf c$ is uniformly close to $\mathsf c_{n,\alpha}$
it does change along $\hat{\gamma}_{n,\alpha}$. Using
$\mathsf c_{n,\alpha}$ instead 
makes the sum $\sum_{k=0}^{m-1}$ appearing in
the statement of Lemma~\ref{lem_est_dA_2} a Birkhoff
sum along the orbit of $\hat{{\mathcal F}}^k$, where only the fast
coordinates $\Pi\hat{{\mathcal F}}^k$ are sampled. This is key
in the approximation by orbits of the standard
billiard map ${\mathcal F}$. It is precisely in this approximation
scheme where we will use the shadowing methods developed in
\cite{MR2034323,MR2031432,MR2241812}, as will be explained next.

Suppose the same notation and setting as in Lemma~\ref{lem_est_dA_2},
and consider one of the terms appearing inside the integral
representation derived in Lemma~\ref{lem_est_dA_2}
\begin{equation*}
  \int_{\hat{\gamma}}
  f( {\mathcal F}\circ \Pi \circ \hat{{\mathcal F}}^k(\hat{x}) )
  \,\hat{\rho}(\hat{x})
  \, \mathrm{m}_{\hat{\gamma}} (d\hat{x})
  \;,
\end{equation*}
where we set
\begin{equation*}
  f( x )
  =
  A'( \mathsf c_{n,\alpha} )
  \,g( x, \mathsf c_{n,\alpha} )
  \;,\quad
  \hat{\gamma} = \hat{\gamma}_{n,\alpha}
  \;,\quad
  \hat{\rho} = \hat{\rho}_{n,\alpha}
\end{equation*}
to shorten the notation.
With the same notation as in 
Item~(\ref{item_lem_growth_lemma_c}) of Lemma~\ref{lem_growth_lemma}
fix two integers $k_1, k_2$ such that
$ c_5\,| \log |\hat{\gamma}| | < k_1 < k_2 < k $,
whose precise values will
be chosen as we go on. Define on $\hat{\gamma}$ a stopping time
$\tau \colon \hat{\gamma} \to {\mathbb N} \cup \{ \infty \}$
by (recall the notation $\hat{\gamma}_j(\hat{x})$ introduced right
before Lemma~\ref{lem_growth_lemma})
\begin{equation*}
  \tau(\hat{x}) = \inf \{ j > k_1 \operatorname{:}
  | \hat{\gamma}_j(\hat{x}) | \geq \zeta_0 \}
  \;,
\end{equation*}
i.e. after $\tau$ many iterations of $\hat{x}$ by $\hat{{\mathcal F}}$
the component containing its image has a length of at least $\zeta_0$,
which will be important when estimating some of the error terms.
It follows from
Item~(\ref{item_lem_growth_lemma_c}) of Lemma~\ref{lem_growth_lemma}
that
\begin{equation*}
   \mathrm{m}_{\hat{\gamma}} \{ \tau \geq k_2 \}
  \leq
  c_6\,\theta_1^{k_2-k_1}\,| \hat{\gamma}|
\end{equation*}
in other words, only on an exponentially small fraction of
$\hat{\gamma}$ the value of $\tau$ is larger that $k_2$.
Denote the connected components of partition of $\hat{\gamma}$
into $ \{ \tau = i \}$, $k_1 <i <k_2$, by
$( \hat{\gamma}^{(i)}_0 )_i$, and denote by
$ ( \lambda^{(i)} )_i $,
$ ( \hat{\rho}^{(i)}_0 )_i $
$ ( \tau^{(i)} )_i$
the corresponding statistical weights, conditional probability 
densities, and $\tau$--values, respectively.
By construction, for every $i$ the image
$\hat{{\mathcal F}}^{\tau^{(i)} }( \hat{\gamma}^{(i)}_0, \hat{\rho}^{(i)}_0 )$
of the standard pair
$( \hat{\gamma}^{(i)}_0, \hat{\rho}^{(i)}_0 )$
consists of a single standard pair, which we will denote by
$( \hat{\gamma}^{(i)}, \hat{\rho}^{(i)} )$.
Then
\begin{equation}
  \label{eqn_sumf_1}
  \begin{split}
    \int_{\hat{\gamma}}
    &
    f( {\mathcal F}\circ \Pi \circ \hat{{\mathcal F}}^k(\hat{x}) )
    \,\hat{\rho}(\hat{x})
    \, \mathrm{m}_{\hat{\gamma}} (d\hat{x})
    =
    \\
    &=
    \sum_{i} \lambda^{(i)}
    \int_{\hat{\gamma}^{(i)}_0}
    f( {\mathcal F}\circ \Pi \circ
    \hat{{\mathcal F}}^{k-\tau^{(i)}}
    \circ
    \hat{{\mathcal F}}^{\tau^{(i)}}(\hat{x})
    )
    \,\hat{\rho}^{(i)}_0(\hat{x})
    \, \mathrm{m}_{\hat{\gamma}^{(i)}_0} (d\hat{x})
    +
    R^{(3)}
    \\
    &=
    \sum_{i} \lambda^{(i)}
    \int_{\hat{\gamma}^{(i)}}
    f( {\mathcal F}\circ \Pi \circ \hat{{\mathcal F}}^{k-\tau^{(i)}}(\hat{x}) )
    \,\hat{\rho}^{(i)}(\hat{x})
    \, \mathrm{m}_{\hat{\gamma}^{(i)}} (d\hat{x})
    +
    R^{(3)}
  \end{split}
\end{equation}
with
\begin{align*}
  | R^{(3)} |
  \leq
  \eta_{\mathrm{max}} \,\left| \,  A'  \, \right|_\infty \,\mathsf c_{n,\alpha}
  \,\Big[ 1 - \sum_{i} \lambda^{(i)} \Big]
  \;,\quad
  1 - \sum_{i} \lambda^{(i)}
  =
   \mathrm{m}_{\hat{\gamma}} \{ \tau \geq k_2 \}
  \leq
  c_6\,\theta_1^{k_2-k_1}\,| \hat{\gamma}|
  \;.
\end{align*}
Next, consider any of the terms in \eqref{eqn_sumf_1} separately,
and note that the identity
\begin{equation}
  \label{eqn_sumf_termSep1}
  \begin{split}
    f( {\mathcal F}\circ \Pi \circ \hat{{\mathcal F}}^{k-\tau^{(i)}}(\hat{x}) )
    &=
    f( {\mathcal F}^{1 + k-\tau^{(i)}} \circ \Pi(\hat{x}) )
    +
    \\
    &\quad
    +
    \sum_{l=1}^{ k-\tau^{(i)} }
    \Big[
    f( {\mathcal F}^{1+ k-\tau^{(i)} -l}\circ
    \Pi \circ \hat{{\mathcal F}}^l(\hat{x}) )
    -
    \\
    &\quad
    \qquad\qquad
    -
    f( {\mathcal F}^{1+ k-\tau^{(i)} -l}\circ
    \Pi\circ P^{-1} \circ \hat{{\mathcal F}}^l(\hat{x}) )
    \Big]
    \;.
  \end{split}
\end{equation}
holds for all $\hat{x}$, we obtain for each of the terms in
\eqref{eqn_sumf_1} the expression
\begin{equation}
  \label{eqn_sumf_2}
  \begin{split}
    \int_{\hat{\gamma}^{(i)}}
    &
    f( {\mathcal F}\circ \Pi \circ \hat{{\mathcal F}}^{k-\tau^{(i)}}(\hat{x}) )
    \,\hat{\rho}^{(i)}(\hat{x})
    \, \mathrm{m}_{\hat{\gamma}^{(i)}} (d\hat{x})
    \\
    &=
    \int_{\hat{\gamma}^{(i)}}
    f( {\mathcal F}^{1 + k-\tau^{(i)}} \circ \Pi(\hat{x}) )
    \,\hat{\rho}^{(i)}(\hat{x})
    \, \mathrm{m}_{\hat{\gamma}^{(i)}} (d\hat{x})
    \\
    &\quad
    +
    \sum_{l=1}^{ k-\tau^{(i)} }
    \int_{\hat{\gamma}^{(i)}}
    \Big[
    f( {\mathcal F}^{1+ k-\tau^{(i)} -l}\circ
    \Pi \circ \hat{{\mathcal F}}^l(\hat{x}) )
    -
    \\
    &\quad
    \qquad\qquad\qquad
    -
    f( {\mathcal F}^{1+ k-\tau^{(i)} -l}\circ
    \Pi\circ P^{-1} \circ \hat{{\mathcal F}}^l(\hat{x}) )
    \Big]
    \,\hat{\rho}^{(i)}(\hat{x})
    \, \mathrm{m}_{\hat{\gamma}^{(i)}} (d\hat{x})
    \;.
  \end{split}
\end{equation}

The following Lemma~\ref{lem_eqn_sumf_2_term1} provides an
approximation for the first of the two terms on the right-hand-side
of \eqref{eqn_sumf_2}.
Recall that
$
\nu(ds, d\varphi)
=
\frac{ \cos\varphi\,d\varphi\,ds }{2\,|\partial Q|}
$
is the invariant measure for the billiard map ${\mathcal F}$.
\begin{lemma}
  \label{lem_eqn_sumf_2_term1}
  For any $\eta_{\mathrm{max}}^*$, $\eta_{1,\mathrm{max}}^*$, $\eta_{2,\mathrm{max}}^*$,
  there exist $\theta < 1$, $0<C$ such that
  \begin{align*}
    \Big|
    \int_{\hat{\gamma}^{(i)}}
    &
    f( {\mathcal F}^{1 + k-\tau^{(i)}} \circ \Pi(\hat{x}) )
    \,\hat{\rho}^{(i)}(\hat{x})
    \, \mathrm{m}_{\hat{\gamma}^{(i)}} (d\hat{x})
    -
    A'( \mathsf c_{n,\alpha} )
    \int_{{\mathcal M}}
    g( x, \mathsf c_{n,\alpha} )
    \,\nu(dx)
    \Big|
    \\
    &\leq
    C
    \,( \eta_{\mathrm{max}} + \eta_{1,\mathrm{max}} )
    \,\left| \,  A'  \, \right|_\infty
    \,\mathsf c_{n,\alpha}
    \,\theta^{ 1 + k-\tau^{(i)} }
  \end{align*}
  uniformly in
  $\eta_{\mathrm{max}}$, $\eta_{1,\mathrm{max}}$, $\eta_{2,\mathrm{max}}$
  provided that they are chosen less than
  $\eta_{\mathrm{max}}^*$, $\eta_{1,\mathrm{max}}^*$, $\eta_{2,\mathrm{max}}^*$,
  respectively.
\end{lemma}
\begin{proof}
  Denote the projection
  $ \Pi (\hat{\gamma}^{(i)}, \hat{\rho}^{(i)} )$
  of the standard pair
  $ (\hat{\gamma}^{(i)}, \hat{\rho}^{(i)} )$
  by
  $ (\gamma^{(i)}, \rho^{(i)} )$, so that
  \begin{align*}
    \int_{\hat{\gamma}^{(i)}}
    f( {\mathcal F}^{1 + k-\tau^{(i)}} \circ \Pi(\hat{x}) )
    \,\hat{\rho}^{(i)}(\hat{x})
    \, \mathrm{m}_{\hat{\gamma}^{(i)}} (d\hat{x})
    =
    \int_{\gamma^{(i)}}
    f( {\mathcal F}^{1 + k-\tau^{(i)}}(x) )
    \,\rho^{(i)}(x)
    \,\mathrm{m}_{\gamma^{(i)}}(dx)
    \;.
  \end{align*}
  Since the length of $ \hat{\gamma}^{(i)} $ is bounded from below
  $| \hat{\gamma}^{(i)} | \geq \zeta_0$, it follows that
  $ |\gamma^{(i)}|$
  is bounded from below by some small fixed constant, say $\delta$.
  The result now follows from the equi-distribution property
  of the billiard map ${\mathcal F}$, e.g. \cite{MR2499824,MR2229799},
  with the error bound given by
  $
  C\,\left\Arrowvert \, f \, \right\Arrowvert_{C^1}\,\theta^{ 1 + k-\tau^{(i)} }
  $
  whenever
  $ 1 + k-\tau^{(i)} \geq K\,| \log |\gamma^{(i)}| |$,
  where $\theta<1$, $C>0$ and $K>0$ are some constants.
  Since the length $|\gamma^{(i)}| \geq \delta$ it suffices to
  have
  $ 1 + k-\tau^{(i)} \geq K'$
  for some fixed constant $K'>0$. At the expense of increasing
  the value of $C$ this condition on
  $ 1 + k-\tau^{(i)} $ can be dropped.
  Recall that we defined
  $
  f( x )
  =
  A'( \mathsf c_{n,\alpha} )
  \,g( x, \mathsf c_{n,\alpha} )
  $.
  With \eqref{eqn_estimates_bFlowSpeedInc} we have
  \begin{equation*}
    \left\Arrowvert \,  f  \, \right\Arrowvert_{C^1}
    \leq
    \left| \,  A'  \, \right|_\infty
    \,\Big[
    \eta_{\mathrm{max}} + \eta_{1,\mathrm{max}}
    +
    \frac{ ( 2 - \eta_{\mathrm{max}} )\,\eta_{\mathrm{max}} }{ 1 - \eta_{\mathrm{max}} }
    \Big]
    \,\mathsf c_{n,\alpha}
    \;,
  \end{equation*}
  and hence we can rewrite the error bound in the claimed form.
\end{proof}

It remains to estimate the second term on the right-hand-side of
\eqref{eqn_sumf_2}. In the proof of
Lemma~\ref{lem_eqn_sumf_2_term1} we made use of the fact that
projection of standard pairs in $\hat{{\mathcal M}}$ are standard
pairs in ${\mathcal M}$, so that well-known results on the
equi-distribution property of the billiard map ${\mathcal F}$ could
be applied. The second term in \eqref{eqn_sumf_2} is different,
and will be estimated by a shadowing argument developed
in \cite{MR2034323,MR2031432,MR2241812}. Here we will not
repeat the fairly lengthy details of this argument, rather
we will explain how it is being used in our present setting.

Consider one of the summands in \eqref{eqn_sumf_2}
\begin{align*}
  S
  =
  \int_{\hat{\gamma}^{(i)}}
  &
  \Big[
  f( {\mathcal F}^{1+ k-\tau^{(i)} -l}\circ
  \Pi \circ \hat{{\mathcal F}}^l(\hat{x}) )
  -
  \\
  &\quad\qquad
  -
  f( {\mathcal F}^{1+ k-\tau^{(i)} -l}\circ
  \Pi\circ P^{-1} \circ \hat{{\mathcal F}}^l(\hat{x}) )
  \Big]
  \,\hat{\rho}^{(i)}(\hat{x})
  \, \mathrm{m}_{\hat{\gamma}^{(i)}} (d\hat{x})
\end{align*}
which appear in the second term on the right-hand-side of
\eqref{eqn_sumf_2}. Note that both terms in the integrand
depend on the integration variable $\hat{x}$ only through
$\hat{{\mathcal F}}^l(\hat{x})$. That is to say that the integral
can be written as an integral with respect to the standard family
$\hat{{\mathcal G}}^{(i)}_l$
corresponding to the image under $\hat{{\mathcal F}}^l$ of the
standard pair $ ( \hat{\gamma}^{(i)}, \hat{\rho}^{(i)} )$
\begin{align*}
  S
  =
  \int_{{\mathcal A}^{(i)}_l}
  \int_{\hat{\gamma}^{(i)}_{l,\alpha} }
  &
  \Big[
  f( {\mathcal F}^{1+ k-\tau^{(i)} -l}\circ\Pi(\hat{x}) )
  -
  \\
  &\quad
  -
  f( {\mathcal F}^{1+ k-\tau^{(i)} -l}\circ
  \Pi\circ P^{-1}(\hat{x}) )
  \Big]
  \,\hat{\rho}^{(i)}_{l,\alpha}(\hat{x})
  \, \mathrm{m}_{ \hat{\gamma}^{(i)}_{l,\alpha}} (d\hat{x})
  \,\lambda^{(i)}_l(d\alpha)
  \;.
\end{align*}
As in the proof of Lemma~\ref{lem_eqn_sumf_2_term1}
set
$
(\gamma^{(i)}_{l,\alpha}, \rho^{(i)}_{l,\alpha} )
=
\Pi (\hat{\gamma}^{(i)}_{l,\alpha}, \hat{\rho}^{(i)}_{l,\alpha} )
$, but due to the presence of $P^{-1}$ in the integrand
we cannot write the above integrals in terms of
$(\gamma^{(i)}_{l,\alpha}, \rho^{(i)}_{l,\alpha} )$ only.
However, since
$\hat{\gamma}^{(i)}_{l,\alpha}$ can be parametrized by
$s$, and $\Pi$ as well as $P^{-1}$ leave
$s$ unchanged we see that
both curves
\begin{equation*}
  \gamma^{(i)}_{l,\alpha}
  =
  \Pi\hat{\gamma}^{(i)}_{l,\alpha}
  \quad\text{and}\quad
  \gamma^{(i,2)}_{l,\alpha}
  =
  \Pi\circ P^{-1}\hat{\gamma}^{(i)}_{l,\alpha}
\end{equation*}
can be parametrized by $s$ (over the exact same
domain).
In particular, the induced mapping
\begin{equation*}
  \Phi^{(i)}_{l,\alpha}
  \colon
  \gamma^{(i)}_{l,\alpha}
  \to
  \gamma^{(i,2)}_{l,\alpha}
  \quad\text{with}\quad
  \Phi^{(i)}_{l,\alpha}\circ \Pi(\hat{x})
  =
  \Pi\circ P^{-1}(\hat{x})
  \quad\forall\;\hat{x} \in \hat{\gamma}^{(i)}_{l,\alpha}
\end{equation*}
is well-defined, smooth, and its difference to the identity map
in the $C^1$--norm and in the $C^2$--norm is
$\operatorname{\mathcal{O}}(\eta_{\mathrm{max}}^* + \eta_{1,\mathrm{max}}^*)$,
$\operatorname{\mathcal{O}}(\eta_{\mathrm{max}}^* + \eta_{1,\mathrm{max}}^* + \eta_{2,\mathrm{max}}^*)$,
respectively, uniformly in the parameters.
With this notation in place we have
\begin{equation}
  \label{eqn_expressionS}
  \begin{split}
    S
    =
    \int_{{\mathcal A}^{(i)}_l}
    \int_{\gamma^{(i)}_{l,\alpha} }
    &
    \Big[
    f( {\mathcal F}^{1+ k-\tau^{(i)} -l}(x) )
    -
    f( {\mathcal F}^{1+ k-\tau^{(i)} -l}\circ
    \Phi^{(i)}_{l,\alpha}(x) )
    \Big]
    \times
    \\
    &\quad
    \qquad\qquad
    \times
    \rho^{(i)}_{l,\alpha}(x)
    \,\mathrm{m}_{ \gamma^{(i)}_{l,\alpha}}(dx)
    \,\lambda^{(i)}_l(d\alpha)
    \;.
  \end{split}
\end{equation}
This identity is key in order to apply the shadowing arguments of
\cite{MR2034323,MR2031432,MR2241812}.
\begin{figure}[ht!]
  \centering
  \includegraphics[height=5cm]{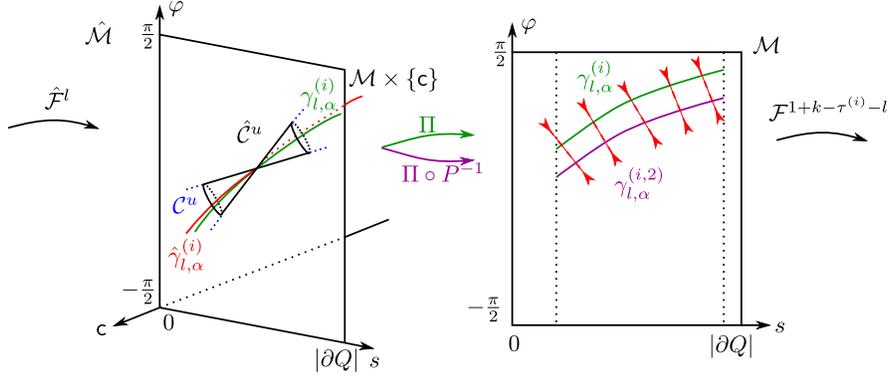}
  \caption{An illustration of the shadowing argument.
  The stable manifolds of ${\mathcal F}$ that provide the coupling
  between the two curves
  $\gamma^{(i)}_{l,\alpha}$ and $\gamma^{(i,2)}_{l,\alpha}$
  in ${\mathcal M}$ are shown on the right side as dashed
  lines in red color.}
  \label{fig_shadowing}
\end{figure}
Indeed, the two curves
$\gamma^{(i)}_{l,\alpha}$ and $\gamma^{(i,2)}_{l,\alpha}$
are $\operatorname{\mathcal{O}}(\eta_{\mathrm{max}}^* + \eta_{1,\mathrm{max}}^*)$--close in ${\mathcal M}$.
Therefore, as long as their lengths, i.e the length of
$\hat{\gamma}^{(i)}_{l,\alpha}$ is not too small, for example
a large enough constant times $(\eta_{\mathrm{max}}^* + \eta_{1,\mathrm{max}}^*)$,
then they are connected by stable manifolds of ${\mathcal F}$ up to
a subset (one each of the two curves) of Lebesgue measure bounded
by a (small) constant times $\operatorname{\mathcal{O}}(\eta_{\mathrm{max}}^* + \eta_{1,\mathrm{max}}^*)$.
The images of these coupled points under
${\mathcal F}^{1+ k-\tau^{(i)} -l}$
have a distance that is bounded by
$
C\,\theta^{1+ k-\tau^{(i)} -l}
\,(\eta_{\mathrm{max}}^* + \eta_{1,\mathrm{max}}^*)
$,
where $C>0$, and $\theta<1$ are some independent constants.
Therefore, a simple $\left\Arrowvert \,  f  \, \right\Arrowvert_{C^1}$ estimate shows that
for these points the integrands in
\eqref{eqn_expressionS} are bounded by
$
C\,\left| \, A' \, \right|_\infty\,\theta^{1+ k-\tau^{(i)} -l}
\,(\eta_{\mathrm{max}}^* + \eta_{1,\mathrm{max}}^*)^2
$. This procedure is illustrated in Fig.~\ref{fig_shadowing}.
Details of how to estimate the residual measure, and the measure
of those curves $\hat{\gamma}^{(i)}_{l,\alpha}$
which are too short, and how to adjust the values of $k_1$, $k_2$
accordingly can be found in
\cite{MR2034323,MR2031432,MR2241812}
in a similar setting.
Combining this with the results of
Lemma~\ref{lem_est_dA_2} and Lemma~\ref{lem_eqn_sumf_2_term1}
we summarize the results in this section in following
Theorem~\ref{thm_avg1}.

Define
\begin{equation}
  \label{eqn_def_AvgInc}
  \bar{g}( \mathsf c )
  =
  \int_{{\mathcal M}}
  g( x, \mathsf c )
  \,\nu(dx)
  \qquad
  \text{for all}\quad
  \mathsf c \geq 0
  \;,
\end{equation}
which is the average (with respect to the invariant measure
$
\nu(ds, d\varphi)
=
\frac{ \cos\varphi\,d\varphi\,ds }{2\,|\partial Q|}
$
of the billiard map ${\mathcal F}$) increment in $\mathsf c_n$.

\begin{theorem}
  \label{thm_avg1}
  For any $\eta_{\mathrm{max}}^*$, $\eta_{1,\mathrm{max}}^*$, $\eta_{2,\mathrm{max}}^*$,
  there exists a $C>0$ such that for all $n$ and all $m\geq 1$
  \begin{align*}
    \Big|
    \int
    &
    \Big(
    A(\mathsf c_{n+m}) - A(\mathsf c_n)
    -
    \sum_{k=0}^{m-1}
    A'( \mathsf c_{n+k} )
    \,\bar{g}( \mathsf c_{n+k} )
    \Big)
    \,\nu_{ {\mathcal G}_0 }(d\hat{x}_0)
    \Big|
    \\
    &\leq
    C
    \,[
    (\eta_{\mathrm{max}}^* + \eta_{1,\mathrm{max}}^* )
    +
    (\eta_{\mathrm{max}}^* + \eta_{1,\mathrm{max}}^* )^2\,m^2
    ]
    \times
    \\
    &\quad \qquad
    \times
    \int
    ( \mathsf c_n \,\left| \,  A''  \, \right|_\infty + \left| \,  A'  \, \right|_\infty )
    \,\mathsf c_n
    \,\nu_{ {\mathcal G}_0 }(d\hat{x}_0)
    \;,
  \end{align*}
  uniformly in
  $\eta_{\mathrm{max}}$, $\eta_{1,\mathrm{max}}$, $\eta_{2,\mathrm{max}}$
  provided that they are chosen less than
  $\eta_{\mathrm{max}}^*$, $\eta_{1,\mathrm{max}}^*$, $\eta_{2,\mathrm{max}}^*$,
  respectively.
\end{theorem}

\section{Averaged dynamics}
\label{sect_avg}

Recall that due to the a-priori bound \eqref{eqn_speed_apriori}
\begin{equation*}
  1 - \eta_{\mathrm{max}}
  \leq
  \frac{ \mathsf c_1 }{ \mathsf c_0 }
  \leq
  1
\end{equation*}
we are primarily interested in describing $(\mathsf c_n)_n$
for values of $n$ between $0$ and $\operatorname{\mathcal{O}}(\eta_{\mathrm{max}}^{-1})$.
To formalize this, we suppose that we are given
\begin{subequations}
  \label{eqn_dissCoeffFamily}
  a family of functions
  \begin{equation}
    (\eta^{(\epsilon)} )_{0 < \epsilon \ll 1 }
  \end{equation}
  such that for some constant $C>0$
  \begin{equation}
    \eta_{\mathrm{max}}^{(\epsilon)}
    + \eta_{1,\mathrm{max}}^{(\epsilon)}
    + \eta_{2,\mathrm{max}}^{(\epsilon)}
    \leq C\,\epsilon
  \end{equation}
  for all $\epsilon \ll 1$, and
  there exists continuous function
  $q \colon [0,\infty) \to {\mathbb R}$, which
  is $C^1$ on $(0,\infty)$, such that
  \begin{equation}
    \frac{1}{\epsilon}\,\eta^{(\epsilon)}(\mathsf c)
    =
    q(\mathsf c)
    +
    \operatorname{\mathcal{O}}(\epsilon)
    \quad\text{as }\epsilon \to 0
  \end{equation}
  uniformly on any compact subinterval of $[0,\infty)$.
\end{subequations}
Naturally we will use a superscript ${ }^{(\epsilon)}$ to
denote the $\epsilon$--dependence expressions that depend
on $\eta^{(\epsilon)}$, e.g.
$ g^{(\epsilon)} $,
$ \bar{g}^{(\epsilon)} $.
However, to avoid cumbersome notation we will not always explicitly
indicate this $\epsilon$-dependence by additional super-scripts
whenever the context is clear enough.

Furthermore, define the function
$h \colon (0,\infty) \to {\mathbb R}$
by
\begin{equation}
  \label{eqn_def_avgSpeedRate}
  h(\mathsf c)
  =
  -
  \mathsf c
  \int_0^{\frac{\pi}{2}}
  q(\mathsf c\,\cos\varphi )
  \,\cos^3\varphi\,d\varphi
  \;,
\end{equation}
and note that
\begin{equation}
  \label{eqn_def_avgSpeedRate_convergence}
  h(\mathsf c)
  =
  \frac{1}{\epsilon}
  \,\bar{g}^{(\epsilon)}( \mathsf c )
  +
  \operatorname{\mathcal{O}}(\epsilon)
\end{equation}
holds uniformly on any compact subinterval of $[0,\infty)$.

To have a specific example at hand, consider either
the case of (small) constant restitution
\begin{subequations}
  \label{eqn_ex1}
  \begin{equation}
    \eta^{(\epsilon)}(\mathsf c)
    =
    \epsilon
    \quad\text{for all}\quad
    \mathsf c > 0
    \;,
  \end{equation}
  in which case
  \begin{equation}
    h(\mathsf c) = - \frac{2}{3}\,\mathsf c
    \;.
  \end{equation}
\end{subequations}
Or more generally, for any given smooth (and increasing)
function $q \colon [0,\infty) \to [0,1]$ with
\begin{subequations}
  \label{eqn_ex2}
  \begin{equation}
    q(0) =0
    \;,\quad
    \quad
    \sup_{\mathsf c>0} q(\mathsf c)\,\mathsf c
    < \infty
    \;,\quad
    \sup_{\mathsf c>0} q(\mathsf c)\,\mathsf c^2
    < \infty
  \end{equation}
  set
  \begin{equation}
    \eta^{(\epsilon)}(\mathsf c)
    =
    \epsilon\,q( \mathsf c^p )
    \quad\text{for some}\quad p>0
    \;,
  \end{equation}
  in which case
  \begin{equation}
    h(\mathsf c)
    =
    -
    \mathsf c
    \int_0^{\frac{\pi}{2}}
    q(\mathsf c^p\,\cos^p\varphi )
    \,\cos^3\varphi\,d\varphi
  \end{equation}
\end{subequations}
In the study of granular media the
choice \eqref{eqn_ex1} is argued to be non-physical, however it
is a common choice for mathematical studies. The choice
\eqref{eqn_ex2} can be derived from elasticity theory, and
as such is considered in applications. We refer
to \cite{MR2101911} for a detailed account on this topic.

In the following, let
$(\eta^{(\epsilon)} )_{0 < \epsilon \ll 1 }$
be as in \eqref{eqn_dissCoeffFamily}.
We follow the standard procedure used in averaging theory
and fix a number
\begin{equation}
  0 < \bar{T} < \infty
  \;,
\end{equation}
and denote by
$\bar{\mathsf c}(\bar{t})$, $0 \leq \bar{t} \leq \bar{T}$,
solutions to the differential equation
\begin{equation}
  \frac{d}{d\bar{t}} \bar{\mathsf c}(\bar{t})
  =
  h( \bar{\mathsf c}(\bar{t}) )
  \;,\quad
  0 \leq \bar{t} \leq \bar{T}
  \;.
\end{equation}

Furthermore, to any of the trajectory
$(\mathsf c_n^{(\epsilon)})_{n}$
(corresponding to $\eta^{(\epsilon)}$)
we associate a continuous function
$\mathsf c^{(\epsilon)} \in C([0,\bar{T}], {\mathbb R})$ by
\begin{equation}
  \label{eqn_def_path}
  \mathsf c^{(\epsilon)}(\bar{t})
  =
  \text{linear interpolation of }
  \mathsf c_{ {\left\lfloor  \epsilon \bar{t} \right\rfloor} }^{(\epsilon)}
  \text{ and }
  \mathsf c_{ {\left\lfloor  \epsilon \bar{t} \right\rfloor} + 1 }^{(\epsilon)}
  \qquad
  0 \leq \bar{t} \leq \bar{T}
  \;.
\end{equation}
In particular, any initial distribution of $\hat{x}_0$ induces
a measure $\Gamma^{(\epsilon)}$ on $C[0,\bar{T}]$.

In the previous sections we derived results that assume
a proper standard family ${\mathcal G}_0$ as distribution
of $\hat{x}_0$. Since the very notion of a standard family
depends on $\epsilon$ through the fact that the cone field
$\hat{{\mathcal C}}^u$ get more narrow in the $\mathsf c$--direction
as $\epsilon$ gets smaller, we will only consider
initial distributions of the form
\begin{equation}
  \label{eqn_initialDist}
  \hat{{\mathcal G}}_0 = {\mathcal G}_0 \times \{ \mathsf c_0 \}
  \;,
\end{equation}
where ${\mathcal G}_0$ is a proper standard family for the billiard
map ${\mathcal F}$. By Lemma~\ref{lem_lift_pG} $ \hat{{\mathcal G}}_0$
is a proper standard family for $\hat{{\mathcal F}}$ for any $\epsilon>0$,
and hence can be used as initial distribution for $\hat{x}_0$
for all $\epsilon$.
We will refer to $\hat{{\mathcal G}}_0$ as a flat proper standard family.

The first step in the study of the limit $\epsilon \to 0$
is the following elementary fact:

\begin{lemma}[Tightness]
  \label{lem_tightness}
  For any flat standard family $\hat{{\mathcal G}}_0$
  the corresponding family
  $( \Gamma^{(\epsilon)} )_\epsilon$
  of measures on $C[0,\bar{T}]$ is tight.
\end{lemma}
\begin{proof}
  Let
  $\hat{{\mathcal G}}_0 = {\mathcal G}_0 \times \{ \mathsf c_0 \}$.
  Then $\mathsf c_n$ is bounded by $\mathsf c_0$ for all $n$.
  Hence the corresponding paths $\bar{\mathsf c}(\bar{t})$ take
  values in $[0,\mathsf c_0]$.
  Furthermore, by \eqref{eqn_estimates_bFlowSpeedInc}
  and \eqref{eqn_dissCoeffFamily} their
  Lipschitz constant is uniformly bounded by
  $C\,\mathsf c_0$. Therefore, all measures
  $ \Gamma^{(\epsilon)} $ are supported in a compact subset
  of $C[0,\bar{T}]$, hence tightness follows.
\end{proof}

Next, for any $X \in C[0,\bar{T}]$ and any
$C^2$--bounded function $A \colon{\mathbb R}\to{\mathbb R}$ define
\begin{equation}
  \label{eqn_def_martingale}
  {\mathsf M}(\bar{t}_0, \bar{t}_1)
  =
  A( X(\bar{t}_1) )
  -
  A( X(\bar{t}_0) )
  -
  \int_{\bar{t}_0}^{\bar{t}_1}
  A'( X(\bar{t}) )
  \,h( X(\bar{t}) )
  \,d\bar{t}
  \;,
\end{equation}
where $0 \leq \bar{t}_0 \leq \bar{t}_1 \leq \bar{T}$.
When appropriate we will also use the notation
${\mathsf M}_{X}(\bar{t}_0, \bar{t}_1)$ and
${\mathsf M}^{A}_{X}(\bar{t}_0, \bar{t}_1)$
to explicitly indicate the dependence on $A$ and $X$.

As was pointed out in the above proof of Lemma~\ref{lem_tightness}
for any $(\mathsf c_n)_n$, the corresponding path
$\mathsf c \in C[0,\bar{T}]$ has a uniformly bounded
Lipschitz constant.
Combining this observation with the result of
Theorem~\ref{thm_avg1} we immediately obtain the following:

\begin{lemma}
  \label{lem_martingale1}
  Fix any $ \mathsf c_* >0$ and $0 < \epsilon_* \ll 1$.
  Then there exists a $C>0$ such that for
  any $0 < \epsilon < \epsilon_*$ and
  any proper standard family ${\mathcal G}_0^\epsilon$ whose
  support is contained in ${\mathcal M} \times (0,\mathsf c_*]$
  the estimate
  \begin{equation*}
    \Big|
    \int
    {\mathsf M}_{ \mathsf c^{(\epsilon)} }(\bar{t}_0, \bar{t}_1)
    \,\nu_{ {\mathcal G}_0^{(\epsilon)} }(d\hat{x}_0)
    \Big|
    \leq
    C\,\sqrt{\epsilon}
    \,( \left| \,  A''  \, \right|_\infty + \left| \,  A'  \, \right|_\infty )
    \;,
  \end{equation*}
  holds for all $0 \leq \bar{t}_0 \leq \bar{t}_1 \leq \bar{T}$.
\end{lemma}
\begin{proof}
  Fix a large integer $M \geq 1$, whose value
  will be chosen below.
  For any integer $0 \leq b_1$, and any integer
  $1 \leq m \leq M$
  the point-wise identity
  \begin{align*}
    A(\mathsf c_{n+b_1 M + m})
    &
    -
    A(\mathsf c_{n})
    -
    \sum_{k=n}^{ n+b_1 M +m -1}
    A'( \mathsf c_{k} )
    \,\bar{g}( \mathsf c_{k} )
    \\
    &=
    \sum_{b = 0}^{b_1 - 1}
    \Big(
    A(\mathsf c_{n+ (b+1) M})
    -
    A(\mathsf c_{n + b M })
    -
    \sum_{k=0}^{ M -1}
    A'( \mathsf c_{n + b M +k} )
    \,\bar{g}( \mathsf c_{n + b M +k} )
    \Big)
    \\
    &\quad
    +
    A(\mathsf c_{n+b_1 M + m})
    -
    A(\mathsf c_{n+b_1 M})
    -
    \sum_{k= n+b_1 M }^{ n+b_1 M +m -1}
    A'( \mathsf c_{k} )
    \,\bar{g}( \mathsf c_{k} )
  \end{align*}
  implies with Theorem~\ref{thm_avg1} the estimate
  \begin{align*}
    \Big|
    \int
    \Big(
    &
    A(\mathsf c_{n+b_1 M+m})
    -
    A(\mathsf c_{n})
    -
    \sum_{k=n}^{ n+ b_1 M+m -1}
    A'( \mathsf c_{k} )
    \,\bar{g}( \mathsf c_{k} )
    \Big)
    \,\nu_{ {\mathcal G}_0 }(d\hat{x}_0)
    \Big|
    \\
    &\leq
    C
    \,\Big[
    \epsilon\,b_1M
    \,( M^{-1} + \epsilon M )
    +
    \epsilon\,( 1 + \epsilon\,m^2)
    \Big] \times
    \\
    &\quad
    \qquad\qquad
    \times
    \int
    ( \mathsf c_n \,\left| \,  A''  \, \right|_\infty + \left| \,  A'  \, \right|_\infty )
    \,\mathsf c_n
    \,\nu_{ {\mathcal G}_0 }(d\hat{x}_0)
    \;,
  \end{align*}
  for some constant $C>0$.

  Now fix $0 \leq \bar{t}_0 < \bar{t}_1 \leq \bar{T}$
  and choose $n$, $b$, $m$ such that
  $ \epsilon\,n \leq \bar{t}_0 < \epsilon\,(n+1)$ and
  $
  \epsilon\,(n+ b_1M+m)
  \leq
  \bar{t}_1
  <
  \epsilon\,(n+ b_1M+m+1)
  $.
  Since the path $\mathsf c \in C[0,\bar{T}]$
  corresponding to $(\mathsf c_n)_n$ is Lipschitz
  continuous we have
  \begin{align*}
    \Big|
    \int
    \Big(
    &
    A(\mathsf c(\bar{t}_1) )
    -
    A(\mathsf c(\bar{t}_0) )
    -
    \int_{\bar{t}_0}^{\bar{t}_1}
    A'( \mathsf c(\bar{t}) )
    \,h(\mathsf c(\bar{t}) )
    \,d\bar{t}
    \Big)
    \,\nu_{ {\mathcal G}_0 }(d\hat{x}_0)
    \Big|
    \\
    &\leq
    C
    \,\epsilon\,(b_1 + 1)\,M
    \,( M^{-1} + \epsilon M )
    \,( \left| \,  A''  \, \right|_\infty + \left| \,  A'  \, \right|_\infty )
    \int
    \,( 1 + \mathsf c_n^2 )
    \,\nu_{ {\mathcal G}_0 }(d\hat{x}_0)
    \;,
  \end{align*}
  for some uniform constant $C>0$ and for any $M$.

  Since
  $
  \epsilon\, b_1 M
  \leq
  \bar{t}_1 - \bar{t}_0
  \leq
  \bar{T}
  $
  regardless of the choice of $M$ and $\epsilon$
  we have
  \begin{equation*}
    \epsilon\,(b_1 + 1)\,M
    \,( M^{-1} + \epsilon M )
    \leq
    ( \bar{T} + \epsilon\,M)
    \,( M^{-1} + \epsilon M )
  \end{equation*}
  for any integer $1 \leq M \leq \epsilon^{-1}\,\bar{T}$.
  Optimizing the choice of $M$
  finishes the proof.
\end{proof}

\begin{theorem}
  \label{thm_avg2}
  Let ${\mathcal G}_0$ be a standard family for the billiard
  map ${\mathcal F}$ on ${\mathcal M}$ with
  ${\mathcal Z}_{{\mathcal G}_0} < \infty$.
  For any $\mathsf c_0>0$ denote by
  $( \Gamma^{(\epsilon)} )_\epsilon$
  the measures on $C[0,\bar{T}]$
  corresponding to the flat standard family
  $\hat{{\mathcal G}}_0 = {\mathcal G}_0 \times \{ \mathsf c_0 \}$
  for $\hat{{\mathcal F}}$ and $0<\epsilon \ll 1$.
  Then
  \begin{align*}
    \limsup_{\epsilon\to 0}
    \Big|
    \int
    {\mathsf M}^{A}_{ X }(\bar{t}_{k+1}, \bar{t}_{k+2})
    \prod_{ i=1 }^k B_i( X(\bar{t}_i) )
    \,\Gamma_{ {\mathcal G}_0 }^{(\epsilon)}(dX)
    \Big|
    =
    0
  \end{align*}
  holds for any
  $0 < \bar{t}_1 < \ldots < \bar{t}_{k+2} \leq \bar{T}$,
  any $C^1$--bounded functions $B_1, \ldots, B_k \colon {\mathbb R}\to{\mathbb R}$,
  and any $C^2$--bounded function $A \colon {\mathbb R}\to{\mathbb R}$.
\end{theorem}
\begin{proof}
  Let
  $
  \Delta
  =
  \min_{i=1, \ldots, k }
  \bar{t}_{i+1} - \bar{t}_i
  $. For every $0 < \epsilon \ll 1$ denote by
  $n_1 < \ldots < n_k < n_{k+1} $
  such that
  $ \epsilon\,n_i \leq \bar{t}_i < \epsilon\,( n_i + 1 ) $
  for all $i=1, \ldots, k+1$.
  Clearly,
  $ n_{i+1} - n_{i} > \epsilon^{-1}\,\Delta - 1 $.
  The point-wise estimate
  \begin{equation*}
    | B_i( \mathsf c(\bar{t}_i) ) - B_i( \mathsf c_{n_i} ) |
    \leq
    C\,\epsilon \,\left| \,  B_i'  \, \right|_\infty \,\mathsf c_*
    \qquad
    i=1, \ldots, k
    \;,
  \end{equation*}
  implies
  \begin{equation*}
    \prod_{ i=1 }^k B_i( \mathsf c(\bar{t}_i) )
    =
    \prod_{ i=1 }^k B_i( \mathsf c_{n_i} )
    +
    R_1
    \;,\quad
    |R_1|
    \leq
    C\,\epsilon
  \end{equation*}
  for some uniform constant $C$.
  Hence
  \begin{align*}
    \Big|
    \int
    {\mathsf M}_{ \mathsf c }(\bar{t}_{k+1}, \bar{t}_{k+2})
    \prod_{ i=1 }^k B_i( \mathsf c(\bar{t}_i) )
    \,\nu_{ {\mathcal G}_0 }(d\hat{x}_0)
    \Big|
    \leq
    | I_1 |
    +
    C\,\epsilon
    \,( \left| \,  A  \, \right|_\infty + \bar{T}\,\left| \,  A'  \, \right|_\infty )
  \end{align*}
  for some uniform $C>0$,
  where
  \begin{equation*}
    I_1
    =
    \int
    {\mathsf M}_{ \mathsf c }(\bar{t}_{k+1}, \bar{t}_{k+2})
    \prod_{ i=1 }^k B_i( \mathsf c_{n_i} )
    \,\nu_{ {\mathcal G}_0 }(d\hat{x}_0)
    \;.
  \end{equation*}

  Since ${\mathcal G}_0$ is  standard family with
  ${\mathcal Z}_{{\mathcal G}_0} < \infty$ it follows from the growth
  property \eqref{eqn_StandardFamilyGrowth} that
  there exists an integer $n_0 \geq 0$ (independent of $\epsilon$)
  such that $ \hat{{\mathcal F}}^{n_0} {\mathcal G}_0 $
  is a proper standard family.
  As we consider only the case $\epsilon \to 0$ we
  may assume that $n_1 > n_0$, and hence
  $ {\mathcal G}_{n_1} = \hat{{\mathcal F}}^{n_1} {\mathcal G}_0 $
  is a proper standard family.

  The next step of the proof is an induction argument.
  Since
  $
  {\mathsf M}_{ \mathsf c }(\bar{t}_{k+1}, \bar{t}_{k+2})
  \prod_{ i=2 }^k B_i( \mathsf c_{n_i} )
  $
  is a bounded function of
  $ (\mathsf c_n)_{n \geq n_2} $,
  it can be written as
  $ Z_{n_2} \circ \hat{{\mathcal F}}^{n_2}(\hat{x}_0) $ for some
  bounded function $Z_{n_2}$ on $\hat{{\mathcal M}}$.

  In particular,
  \begin{align*}
    I_1
    &=
    \int
    Z_{n_2} \circ \hat{{\mathcal F}}^{n_2}(\hat{x}_0)
    \,B_1( \mathsf c_{n_1} )
    \,\nu_{ {\mathcal G}_0 }(d\hat{x}_0)
    =
    \int
    Z_{n_2} \circ \hat{{\mathcal F}}^{n_2-n_1}(\hat{x})
    \,B_1( \mathsf c )
    \,\nu_{ {\mathcal G}_{n_1} }(d\hat{x})
  \end{align*}
  and hence
  \begin{align*}
    I_1
    &=
    \int_{{\mathcal A}_{n_1} }
    \int_{ \hat{\gamma}_{\alpha} }
    Z_{n_2} \circ \hat{{\mathcal F}}^{n_2-n_1}(\hat{x})
    \,B_1( \mathsf c )
    \,\hat{\rho}_{\alpha}( \hat{x} )
    \, \mathrm{m}_{ \hat{\gamma}_{\alpha} } ( d\hat{x} )
    \,\lambda( d\alpha )
    \;.
  \end{align*}
  Furthermore, for any choice of $\mathsf c_{\alpha}$ on
  $\hat{\gamma}_{\alpha}$ we have
  \begin{align*}
    I_1
    &=
    I_2
    +
    R_2
    \;,\quad
    I_2
    =
    \int_{{\mathcal A}_{n_1} }
    B_1( \mathsf c_{\alpha} )
    \int_{ \hat{\gamma}_{\alpha} }
    Z_{n_2} \circ \hat{{\mathcal F}}^{n_2-n_1}(\hat{x})
    \,\hat{\rho}_{\alpha}( \hat{x} )
    \, \mathrm{m}_{ \hat{\gamma}_{\alpha} } ( d\hat{x} )
    \,\lambda( d\alpha )
  \end{align*}
  and
  \begin{align*}
    |R_2|
    &\leq
    \left| \,  Z_{n_2}  \, \right|_\infty
    \int_{{\mathcal A}_{n_1} }
    \int_{ \hat{\gamma}_{\alpha} }
    |
    B_1( \mathsf c )
    -
    B_1( \mathsf c_{\alpha} )
    |
    \,\hat{\rho}_{\alpha}( \hat{x} )
    \, \mathrm{m}_{ \hat{\gamma}_{\alpha} } ( d\hat{x} )
    \,\lambda( d\alpha )
    \\
    &\leq
    C\,\epsilon \,\left| \,  B_1'  \, \right|_\infty \left| \,  Z_{n_2}  \, \right|_\infty
  \end{align*}
  for some uniform constant $C>0$ (where we used
  \eqref{eqn_uCurveExt_variationC}).

  By the growth property
  \eqref{eqn_StandardFamilyGrowth}
  of standard families and
  \eqref{eqn_Z_asymp}
  it follows from Markov's inequality
  that there exists a constant $C>0$ such that
  \begin{align*}
    {\mathcal A}_{n_1}^m
    =
    \Big\{ \alpha \in {\mathcal A}_{n_1} \operatorname{:}
    \hat{{\mathcal F}}^m(\hat{\gamma}_{\alpha}, \hat{\rho}_{\alpha})
    \text{ is a proper standard family }
    \Big\}
  \end{align*}
  satisfies
  \begin{equation*}
    \lambda( {\mathcal A}_{n_1}^m )
    \leq
    C\,\epsilon
    \quad\text{for any}\quad
    m \geq \frac{ \log \frac{1}{\epsilon} }{ \log \frac{1}{\theta} }
  \end{equation*}
  for all $0 < \epsilon \ll 1$.
  And since $ n_2 - n_1 > \epsilon^{-1}\,\Delta - 1 $
  there exist an $0< \epsilon_* \ll 1$ such that
  \begin{equation*}
    \lambda( {\mathcal A}_{n_1} \setminus {\mathcal A}_{n_1}^{n_2-n_1} )
    \leq
    C\,\epsilon
    \quad\text{for all}\quad
    0 < \epsilon < \epsilon_*
    \;.
  \end{equation*}
  For any $\alpha \in {\mathcal A}_{n_1}^{n_2-n_1} $
  \begin{align*}
    \Big|
    \int_{ \hat{\gamma}_{\alpha} }
    Z_{n_2} \circ \hat{{\mathcal F}}^{n_2-n_1}(\hat{x})
    \,\hat{\rho}_{\alpha}( \hat{x} )
    \, \mathrm{m}_{ \hat{\gamma}_{\alpha} } ( d\hat{x} )
    \Big|
    &=
    \Big|
    \int
    Z_{n_2}(\hat{x})
    \,\nu_{
    \hat{{\mathcal F}}^{n_2-n_1}(\hat{\gamma}_{\alpha}, \hat{\rho}_{\alpha})
    }( d\hat{x} )
    \Big|
    \\
    &\leq
    \sup_{{\mathcal G}}
    \Big|
    \int Z_{n_2}(\hat{x}) \,\nu_{{\mathcal G}}( d\hat{x} )
    \Big|
  \end{align*}
  where the supremum is taken over all proper standard families
  ${\mathcal G}$ supported in ${\mathcal M} \times[0,\mathsf c_*]$.
  Therefore,
  \begin{align*}
    | I_2 |
    &\leq
    C\,\epsilon
    \,\left| \,  B_1  \, \right|_\infty
    \,\left| \,  Z_{n_2}  \, \right|_\infty
    +
    \left| \,  B_1  \, \right|_\infty
    \sup_{{\mathcal G}}
    \Big|
    \int Z_{n_2}(\hat{x}) \,\nu_{{\mathcal G}}( d\hat{x} )
    \Big|
    \;.
  \end{align*}
  
  Proceeding by induction over $k$ we conclude that
  there exist $\epsilon_*>0$, and constants
  $0<C_1,C_2$ (depending on the $C^1$--norm of $A$
  and $B_1, \ldots, B_k$) such that
  \begin{align*}
    \Big|
    \int
    {\mathsf M}_{ \mathsf c }(\bar{t}_{k+1}, \bar{t}_{k+2})
    \prod_{ i=1 }^k B_i( \mathsf c(\bar{t}_i) )
    \,\nu_{ {\mathcal G}_0 }(d\hat{x}_0)
    \Big|
    \leq
    C_1 \epsilon
    +
    C_2
    \sup_{{\mathcal G}}
    \Big|
    \int
    Z \circ \hat{{\mathcal F}}^{n_{k+1} - n_k}(\hat{x})
    \,\nu_{{\mathcal G}}( d\hat{x} )
    \Big|
  \end{align*}
  for all $0<\epsilon<\epsilon_*$,
  where
  $
  {\mathsf M}_{ \mathsf c }(\bar{t}_{k+1}, \bar{t}_{k+2})
  =
  Z \circ \hat{{\mathcal F}}^{n_{k+1}}(\hat{x}_0)
  $
  for some bounded $Z \colon \hat{{\mathcal M}} \to {\mathbb R}$.

  To finish the proof, notice that the estimate provided by
  Lemma~\ref{lem_martingale1}
  is uniform in the standard family chosen as initial condition,
  hence
  \begin{equation*}
    \sup_{{\mathcal G}}
    \Big|
    \int
    Z \circ \hat{{\mathcal F}}^{n_{k+1} - n_k}(\hat{x})
    \,\nu_{{\mathcal G}}( d\hat{x} )
    \Big|
    \leq
    C_3\,\sqrt{\epsilon}
    \;.
  \end{equation*}
  In terms of the induced measure $\Gamma$ on
  $C[0,\bar{T}]$
  \begin{align*}
    \int
    {\mathsf M}_{ \mathsf c }(\bar{t}_{k+1}, \bar{t}_{k+2})
    \prod_{ i=1 }^k B_i( \mathsf c(\bar{t}_i) )
    \,\nu_{ {\mathcal G}_0 }(d\hat{x}_0)
    =
    \int
    {\mathsf M}_{ X }(\bar{t}_{k+1}, \bar{t}_{k+2})
    \prod_{ i=1 }^k B_i( X(\bar{t}_i) )
    \,\Gamma_{ {\mathcal G}_0 }^{(\epsilon)}(dX)
  \end{align*}
  which completes the proof.
\end{proof}

Fix a flat standard family $\hat{{\mathcal G}}$, and denote
by $\Gamma_{ \hat{{\mathcal G}} }^{(\epsilon)}$ the
family of measures on $C[0,\bar{T}]$ induced by
$\hat{{\mathcal F}}$. By Lemma~\ref{lem_tightness} this sequence is tight,
and from Theorem~\ref{thm_avg2} we see that any limit point
$\Gamma_*$ satisfies
\begin{align*}
  \int
  \Big(
  A( X(\bar{t}_{k+2}) )
  -
  &
  A( X(\bar{t}_{k+1}) )
  -
  \int_{\bar{t}_{k+1}}^{\bar{t}_{k+2}}
  A'( X(\bar{t}) )
  \,h( X(\bar{t}) )
  \,d\bar{t}
  \Big)
  \times
  \\
  &\times
  \prod_{ i=1 }^k B_i( X(\bar{t}_i) )
  \,\Gamma_*(dX)
  =
  0
\end{align*}
for any
$0 < \bar{t}_1 < \ldots < \bar{t}_{k+2} \leq \bar{T}$,
any $C^1$--bounded functions $B_1, \ldots, B_k \colon {\mathbb R}\to{\mathbb R}$,
and any $C^2$--bounded function $A \colon {\mathbb R}\to{\mathbb R}$.
But this means that $\Gamma_*$ solves
the martingale problem \cite{MR2190038} for corresponding to
the linear operator
$
{\mathcal L} A (\mathsf c)
=
A'( \mathsf c ) \,h( \mathsf c )
$
with initial condition concentrated on $\{ \mathsf c_0 \}$.
Clearly, this martingale problem has a unique solution, namely
the measure
$\Gamma_{\mathsf c_0}$
on $C[0,\bar{T}]$ concentrated on the solution
curve to the initial value problem
\begin{equation}
  \label{eqn_IVP_c}
  \frac{d}{d\bar{t}} \bar{\mathsf c}(\bar{t})
  =
  h( \bar{\mathsf c}(\bar{t}) )
  \;,\quad
  \bar{\mathsf c}(0) = \mathsf c_0
  \;,\quad
  0 \leq \bar{t} \leq \bar{T}
  \;.
\end{equation}
Therefore, the limit point $\Gamma_*$ is unique, and hence
the family $\Gamma_{ \hat{{\mathcal G}} }^{(\epsilon)}$ actually converges
weakly to $\Gamma_{\mathsf c_0}$ as $\epsilon$ is sent to $0$.

Let us point out that so far we made the
assumption that the billiard table
$Q$ has a piece-wise smooth boundary with finite horizon,
i.e. $\tau_{\mathrm{max}} < \infty$, and also $\tau_{\mathrm{min}}>0$.
The latter would rule out tables that are not on the torus,
because the presence of a corner point of the boundary
$\partial Q$ would clearly violate that condition.
However, if we assume that the boundary of $Q$ has no
cusps, then there can be at most finitely many corner points.
Therefore, there exists an integer $n_*>0$ and $\tau_*>0$
such that in any sequence of $n_*$ consecutive reflections
at least one free path is longer than $\tau_*$. Therefore,
the results of Section~\ref{sect_cones} and Section~\ref{sect_ucurves}
carry over to billiard tables with finite horizon without
cusps with little or no modification. Indeed, the central result
of those sections was the Lemma~\ref{lem_growth_lemma}, which holds
as stated in the more general setting.
Therefore, we obtain the following:

\begin{theorem}[Averaged dynamics]
  \label{thm_avg3}
  Let $Q$ be a dispersing billiard table with piece-wise
  smooth boundary with finite horizon and no cusps.
  Let ${\mathcal G}_0$ be a standard family for the billiard
  map ${\mathcal F}$ on ${\mathcal M}$ with
  ${\mathcal Z}_{{\mathcal G}_0} < \infty$.
  For any $\mathsf c_0>0$ consider the
  flat standard family
  $\hat{{\mathcal G}}_0 = {\mathcal G}_0 \times \{ \mathsf c_0 \}$
  for $\hat{{\mathcal F}}$ and $0<\epsilon \ll 1$.
  Then
  \begin{equation*}
    \lim_{\epsilon\to 0}
    \int
    \sup_{0 \leq \bar{t} \leq \bar{T}}
    |
    \mathsf c^{(\epsilon)}(\bar{t})
    -
    \bar{\mathsf c}(\bar{t})
    |
    \,\nu_{{\mathcal G}_0}(dx_0)
    =
    0
    \;,
  \end{equation*}
  where $\bar{\mathsf c}$ denotes the solution to \eqref{eqn_IVP_c}.
\end{theorem}

The result of Theorem~\ref{thm_avg3} shows that the sequence
$(\mathsf c_n)_n$, for $0 \leq n \leq \epsilon^{-1}\,\bar{T}$,
is well approximated by the solution
to the initial value problem \eqref{eqn_IVP_c} with
initial value being $\mathsf c_0$.
From the point of view of the application we have in mind, this
result is not quite satisfactory, because the
above mentioned approximation is in terms of the so-called
collision times, not the real time that has elapsed.

To address this issue we first point out that the real time
elapsed between two consecutive collisions
$\hat{x}_n$ and $\hat{x}_{n+1}$ is given by
$ \frac{\tau(x_n)}{\mathsf c_n} $, where
$\tau(x_n)$ denotes the free path of the billiard
map ${\mathcal F}$, which is determined by the geometry of the billiard
table $Q$. In order to consider
$\frac{\bar{T}}{\epsilon}$--many collisions
(as in Theorem~\ref{thm_avg3}) and have a total increment of the real
time of order one, we scale the time increment by $\epsilon$.
Hence we consider the joint dynamics
\begin{equation}
  \label{eqn_jointDynamics}
  \hat{x}_{n+1} = \hat{{\mathcal F}}( \hat{x}_n )
  \;,\quad
  t_{n+1} = t_n + \epsilon\, \frac{\tau(x_n)}{\mathsf c_n}
\end{equation}
with $t_n$ denoting the moment in time of the $n$--th collision.

Clearly, we would like a generalization of
Theorem~\ref{thm_avg3} that also includes and approximation
of $(t_n)_n$.
There is a significant difference between the analysis of the
joint dynamics \eqref{eqn_jointDynamics} and our previous
analysis. The reason why our analysis of the evolution of
$\mathsf c_n$ was rather involved is the fact that although
$(\mathsf c_n)_n$ changes only very slowly, the dynamics
of the fast variable $(x_n)_n$ depends on it, i.e. the
joint dynamics is fully coupled. Because of this we had to
study the joint dynamics of
$\hat{x}_n = (x_n, \mathsf c_n)$.
Augmenting now the evolution of $(t_n)_n$ is significantly less
complicated, because the values of $(t_n)_n$ are computed
along an orbit $(\hat{x}_n)_n$ without changing the dynamics
of $(\hat{x}_n)_n$. In particular, analyzing the joint dynamics of
$(\hat{x}_n, t_n)_n$ does not require us to construct invariant
cones and related invariant structures on the joint state space.
Instead, a straightforward adaptation of the methods of
Section~\ref{sect_avg} to paths
$(\mathsf c(\bar{t}), t(\bar{t}))_{0 \leq \bar{t} \leq \bar{T}}$
in $C[0,\bar{T}]$ yields an extension of Theorem~\ref{thm_avg3}
that we simply state below in form of Theorem~\ref{thm_avg4} without proof.

In order to state the averaged dynamics of
$(\mathsf c(\bar{t}), t(\bar{t}))$
we recall the average of the free path of the billiard dynamics
can be expressed \cite{MR2229799} in terms of basic geometric
properties of the billiard table $Q$
\begin{equation*}
  \int \tau(x)\,\nu(dx)
  =
  \frac{ \pi |Q| }{ |\partial Q| }
\end{equation*}
so that we consider the following initial value problem
\begin{equation}
  \label{eqn_IVP_ct}
  \begin{split}
    \frac{d}{d\bar{t}} \bar{\mathsf c}(\bar{t})
    &=
    h( \bar{\mathsf c}(\bar{t}) )
    \;,\quad
    \bar{\mathsf c}(0) = \mathsf c_0
    \\
    \frac{d}{d\bar{t}} t(\bar{t})
    &=
    \frac{ \pi |Q| }{ |\partial Q| }
    \,\frac{1}{ \mathsf c(\bar{t})}
    \;,\quad
    t(0) = 0
  \end{split}
  \qquad
  0 \leq \bar{t} \leq \bar{T}
  \;,
\end{equation}
generalizing \eqref{eqn_IVP_c}.

\begin{theorem}[Averaged joint dynamics]
  \label{thm_avg4}
  Let $Q$ be a dispersing billiard table with piece-wise
  smooth boundary with finite horizon and no cusps.
  Let ${\mathcal G}_0$ be a standard family for the billiard
  map ${\mathcal F}$ on ${\mathcal M}$ with
  ${\mathcal Z}_{{\mathcal G}_0} < \infty$.
  For any $\mathsf c_0>0$ consider the
  flat standard family
  $\hat{{\mathcal G}}_0 = {\mathcal G}_0 \times \{ \mathsf c_0 \}$
  for $\hat{{\mathcal F}}$ and $0<\epsilon \ll 1$.
  Then
  \begin{equation*}
    \lim_{\epsilon\to 0}
    \int
    \Big(
    \sup_{0 \leq \bar{t} \leq \bar{T}}
    |
    \mathsf c^{(\epsilon)}(\bar{t})
    -
    \bar{\mathsf c}(\bar{t})
    |
    +
    \sup_{0 \leq \bar{t} \leq \bar{T}}
    |
    t^{(\epsilon)}(\bar{t})
    -
    t(\bar{t})
    |
    \Big)
    \,\nu_{{\mathcal G}_0}(dx_0)
    =
    0
    \;,
  \end{equation*}
  where $(\bar{\mathsf c}, t)$ denotes the solution to \eqref{eqn_IVP_ct}.
\end{theorem}

Theorem~\ref{thm_avg4} shows that in the limit as $\epsilon$ tends to
$0$ the joint dynamics
$( \mathsf c^{(\epsilon)}(\bar{t}), t^{(\epsilon)}(\bar{t}) )$
in $C[0,\bar{T}]$ can be well approximated by the solution
to the initial value problem \eqref{eqn_IVP_ct}.
And that initial value problem implies
\begin{equation*}
  \frac{d}{dt} \bar{\mathsf c}(t)
  =
  \frac{ |\partial Q| }{ \pi |Q| }
  \, \bar{\mathsf c}(t)
  \,h( \bar{\mathsf c}(t) )
  \;,\quad
  \bar{\mathsf c}(0) = \mathsf c_0
  \;,
\end{equation*}
which eliminates the artificial variable $\bar{t}$,
and expresses the evolution of $\bar{\mathsf c}$
in terms of the elapsed time $t$.
Using the definition \eqref{eqn_def_avgSpeedRate}
of $h$ we can rewrite this as
\begin{equation}
  \label{eqn_IVP_c_t}
  \frac{d}{dt} \bar{\mathsf c}(t)
  =
  -
  \frac{ |\partial Q| }{ \pi |Q| }
  \, \bar{\mathsf c}(t)^2
  \int_0^{\frac{\pi}{2}}
  q(\bar{\mathsf c}(t)\,\cos\varphi )
  \,\cos^3\varphi\,d\varphi
  \;,\quad
  \bar{\mathsf c}(0) = \mathsf c_0
  \;.
\end{equation}
directly in terms of $q$.
This proves our main result Theorem~\ref{thm_mainResult}.

\section{Conclusion}
\label{sect_conclusions}.

The derivation of transport coefficients from
microscopic models typically results in
an expression for the transport coefficient in
terms of a correlation sum typically referred to
as Green-Kubo formula
\cite{zbMATH00052458},
\cite{MR606459,MR1138952,MR1149489,MR1376436,CHERNOV200037,MR1224092,MR1832968,MR2389891,MR2349520,MR2583573,MR2737493}.
The present work derives an equation for the cooling of
a system with dissipative interactions, which is not
expressed through a Green-Kubo formula. This is because the effect
we study is due to the slow motion being averaged by the
fast moving billiard dynamics. The main result is the derivation
of Haff's law for the cooling. Indeed, in the special case
of a constant restitution coefficient $\eta\equiv \epsilon$
it follows from \eqref{eqn_ex1} that the statement of
Theorem~\ref{thm_mainResult} takes on the particular form
\begin{equation*}
  \frac{d}{dt} \bar{\mathsf c}(t)
  =
  -
  \frac{2}{3}
  \,\frac{ |\partial Q| }{ \pi |Q| }
  \, \bar{\mathsf c}(t)^2
  \;,
\end{equation*}
whose solutions read
\begin{equation*}
  \frac{1}{ \bar{\mathsf c}(t) }
  =
  \frac{1}{ \bar{\mathsf c}(0) }
  +
  \frac{2}{3}
  \,\frac{ |\partial Q| }{ \pi |Q| }
  \,t
  \;.
\end{equation*}
In other words, as a function of time
the reciprocal of the speed (i.e. the square-root
of the internal kinetic energy) is a straight line.
This is precisely Haff's cooling law \cite{2150608,MR2101911}.

Our assumption of a finite horizon is of technical nature.
It is used in two places. First it is used in the derivation
of the growth lemma through the one-step expansion
property Lemma~\ref{lem_oneStep_expansion}.  For standard billiards
and certain perturbations of it
this property is known to be true also for the infinite horizon
situation \cite{MR2349520,MR2583573,MR2737493}.
The second place where the finite horizon assumption was used
is the extension of Theorem~\ref{thm_avg3} to Theorem~\ref{thm_avg4}.
In both places it is very likely true that the finite
horizon condition is not needed.

A significantly more complicated extension
of our results would be a generalization to many particles. No
results related to this are known to the author.

\bibliography{haff_law}
\bibliographystyle{plain}

\end{document}